\documentclass[reqno]{amsart}

\usepackage{latexsym,amssymb,amsmath, amsfonts}
\usepackage{amsthm}
\usepackage{amscd}
\usepackage[shortlabels]{enumitem}
\usepackage{a4wide}

\usepackage[backref=page]{hyperref}

\usepackage{graphicx}
\usepackage{import}
\usepackage{adjustbox}
\numberwithin{equation}{section}
\usepackage[dvipsnames,svgnames,table]{xcolor}
\usepackage{hyperref}
\hypersetup{colorlinks=true,
	linkcolor=blue,
	urlcolor=blue,
	citecolor=red} 
\newcommand{\R}{\mathbb{R}}
\newcommand{\C}{\mathbb{C}}
\newcommand{\Z}{\mathbb{Z}}

\renewcommand{\le}{\leqslant}
\renewcommand{\ge}{\geqslant}
\renewcommand{\leq}{\leqslant}
\renewcommand{\geq}{\geqslant}

\newcommand{\be}{\begin{equation}}
\newcommand{\en}{\end{equation}}
\newcommand{\ee}{\end{equation}}

\newcommand{\bt}{\begin{theorem}}
\newcommand{\et}{\end{theorem}}
\newcommand{\bp}{\begin{proof}}
\newcommand{\ep}{\end{proof}}
\newcommand{\bc}{\begin{cor}}
\newcommand{\ec}{\end{cor}}
\newcommand{\bl}{\begin{lemma}}
\newcommand{\el}{\end{lemma}}
\newcommand{\bprop}{\begin{prop}}
\newcommand{\eprop}{\end{prop}}

\newtheorem{theorem}{Theorem}[section]
\newtheorem{remark}{Remark}
\newtheorem{lemma}[theorem]{Lemma}
\newtheorem{definition}{Definition}
\newtheorem{proposition}[theorem]{Proposition}
\newtheorem{con}{Conjecture}
\newtheorem{example}[con]{Example}

\numberwithin{theorem}{section} \numberwithin{definition}{section}

\newcommand{\RNum}[1]{\uppercase\expandafter{\romannumeral #1\relax}}
\def\R{\mathbb{R}}

\def\Z{\mathbb{Z}}
\def\C{\mathbb{C}}

\newcommand{\bPhi}{\boldsymbol{\Phi}}
\newcommand{\bU}{\mathbf{U}}

\newcommand{\sech}{\operatorname{sech}}

\newcommand{\vertiii}[1]{{\left\vert\kern-0.25ex\left\vert\kern-0.25ex\left\vert #1 
		\right\vert\kern-0.25ex\right\vert\kern-0.25ex\right\vert}}

\theoremstyle{definition}

\author[J. Angulo]{Jaime Angulo Pava}

\author[A. Mu\~noz]{Alexander Mu\~noz}

\date{}
\title[Airy and Schr\"odinger-type equations on looping-edge graphs]{Airy and Schr\"odinger-type equations on looping-edge graphs and applications.}

\begin{document}

\maketitle

\centerline{ Department of Mathematics,
IME-USP}
 \centerline{Rua do Mat\~ao 1010, Cidade Universit\'aria, CEP 05508-090,
 S\~ao Paulo, SP (Brazil)}
 \centerline{\tt angulo@ime.usp.br, alexd@usp.br}

\section*{Abstract}

The aim of this work is to study the Airy and Schrödinger operators on 
looping-edge graphs, a class of metric graphs consisting of a circle and 
a finite number $N$ of infinite half-lines attached to a common vertex. 
For the Airy operator, we characterize all extensions generating unitary 
and contractive dynamics in terms of self-orthogonal subspaces and linear 
operators acting on indefinite inner product spaces (Krein spaces) 
associated to the boundary values at the vertex. Employing similar
abstract techniques, we then describe a systematic way to produce self-adjoint extensions of the Schr\"odinger operator that are compatible with prescribed boundary relations on looping-edge and $\mathcal{T}$-shaped graphs.

\qquad\\
\textbf{Mathematics  Subject  Classification (2020)}. Primary
35Q51, 35Q55, 81Q35, 35R02; Secondary 47E05.\\
\textbf{Key  words}.  Schr\"odinger equation, Airy equation, quantum graphs, Krein spaces, extension theory of symmetric operators.

\section{Introduction}

Nonlinear evolution models on metric graphs arise as models in wave 
propagation, for instance, in quasi-one-dimensional systems (e.g., 
meso- or nanoscale) that resemble a thin neighborhood of a graph. These 
models provide a convenient means to study various physical effects in 
the real world, both from theoretical and experimental perspectives. The 
flexibility in setting the geometry of the graph configuration allows for 
the creation of different dynamics. To mention a few examples of 
bifurcated systems: Josephson junction structures, networks of planar 
optical waveguides and fibers, branched structures associated with DNA, 
blood pressure waves in large arteries, nerve impulses in complex arrays 
of neurons, conducting polymers, and Bose-Einstein condensation (see 
\cite{BK, BlaExn08, BurCas01, Chuiko, Crepeau, Ex, Fid15, K, Mug15, 
Noj14, SBM} and references therein).

Recently, nonlinear models on graphs, such as the nonlinear 
Schr\"odinger equation, the sine-Gordon model, and the Korteweg-de Vries 
model, have been studied extensively (see \cite{AdaNoj14, AngGol17a, 
AngGol17b, AC, AC1, AP1, AP2, AP3, AST, Fid15, Noj14} and references 
therein). From a mathematical viewpoint, an evolution model on a graph is 
equivalent to a system of PDEs defined on the edges (intervals), where 
the coupling is determined exclusively by the boundary conditions at the 
vertices (known as the ``topology of the graph''), which governs the 
dynamics on the network via groups. This area of study has attracted 
significant attention, particularly in the context of soliton transport. 
Solitons and other nonlinear waves in branched systems provide valuable 
insights into the dynamics of these models.

Addressing these issues is challenging because both the equations of 
motion and the graph topology can be complex. Additionally, a critical 
aspect that complicates the analysis is the presence of one or more 
vertices where a soliton profile, arriving at the vertex along one of 
the edges, undergoes complicated behaviors such as reflection and the 
emergence of radiation. This makes it difficult to observe how energy 
travels across the network. As a result, the study of soliton 
propagation through networks presents significant challenges. The 
mechanisms for the existence and stability (or instability) of soliton 
profiles remain unclear for many types of graphs and models.

We recall that a metric graph $\mathcal{G}$ is a structure represented 
by a finite number of vertices $V=\{\nu_i\}$ and a set of adjacent edges 
at the vertices $E=\{e_j\}$ (for further details, see \cite{BK}). Each 
edge $e_j$ can be identified with a finite or infinite interval of the 
real line, $I_e$. Thus, the edges of $\mathcal{G}$ are not merely 
abstract relations between vertices but can be viewed as physical 
``wires'' or ``networks'' connecting them. The notation $e \in E$ will 
be used to indicate that $e$ is an edge of $\mathcal{G}$. This 
identification introduces the coordinate $x_e$ along the edge $e$.

An example of a metric graph, which will be the main focus of our study here, 
is a looping-edge graph. This is a graph composed of a circle and a 
finite number $N$ of infinite half-lines attached to a common vertex 
$\nu = L$. If the circle is identified with the interval $[-L, L]$ and 
the half-lines with $[L, \infty)$, we obtain a particular metric graph 
structure that we denote with $\mathcal{G}$, represented by $E = 
\{[-L, L], [L, \infty), \dots, [L, \infty)\}$ and $V = \{L\}$ (see 
Figure~\ref{Fig1}). We use the subscripts $0$ and $j$ to refer to the 
edges $e_0 = [-L, L]$ and $e_j = [L, \infty)$, $j = 1, \dots, N$, 
respectively. The special case $N = 1$ is known as a tadpole graph 
(see also Figure~\ref{Fig1}).

\begin{figure}
    \centering
    \resizebox{\linewidth}{!}{
\begingroup%
  \makeatletter%
  \providecommand\color[2][]{%
    \errmessage{(Inkscape) Color is used for the text in Inkscape, but the package 'color.sty' is not loaded}%
    \renewcommand\color[2][]{}%
  }%
  \providecommand\transparent[1]{%
    \errmessage{(Inkscape) Transparency is used (non-zero) for the text in Inkscape, but the package 'transparent.sty' is not loaded}%
    \renewcommand\transparent[1]{}%
  }%
  \providecommand\rotatebox[2]{#2}%
  \newcommand*\fsize{\dimexpr\f@size pt\relax}%
  \newcommand*\lineheight[1]{\fontsize{\fsize}{#1\fsize}\selectfont}%
  \ifx\svgwidth\undefined%
    \setlength{\unitlength}{595.27559055bp}%
    \ifx\svgscale\undefined%
      \relax%
    \else%
      \setlength{\unitlength}{\unitlength * \real{\svgscale}}%
    \fi%
  \else%
    \setlength{\unitlength}{\svgwidth}%
  \fi%
  \global\let\svgwidth\undefined%
  \global\let\svgscale\undefined%
  \makeatother%
  \begin{picture}(1,0.35595278)%
    \lineheight{1}%
    \setlength\tabcolsep{0pt}%
    \put(0,0){\includegraphics[width=\unitlength,page=1]{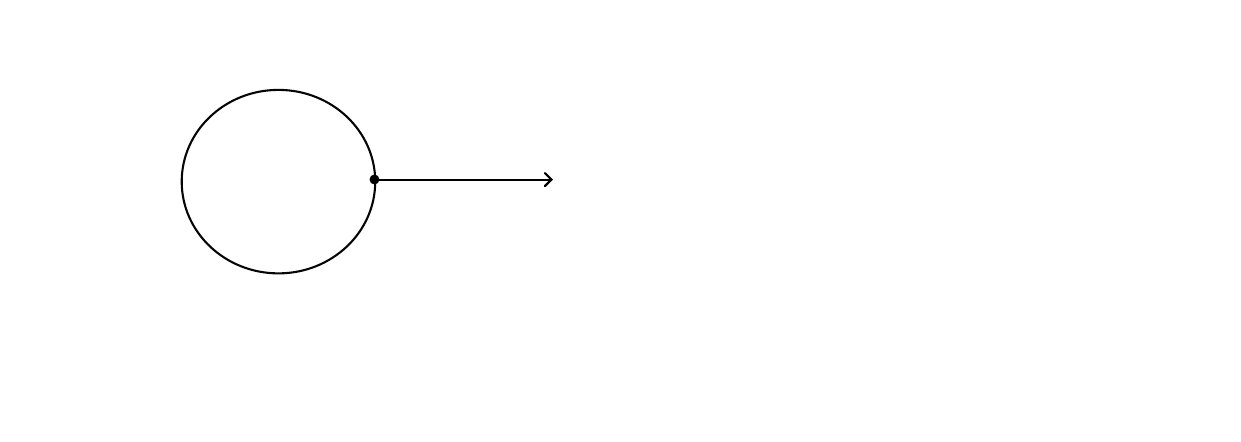}}%
    \put(0.30484472,0.22370066){\color[rgb]{0,0,0}\makebox(0,0)[lt]{\lineheight{1.25}\smash{\begin{tabular}[t]{l}$-L$\end{tabular}}}}%
    \put(0.30858678,0.18552168){\color[rgb]{0,0,0}\makebox(0,0)[lt]{\lineheight{1.25}\smash{\begin{tabular}[t]{l}$L$\end{tabular}}}}%
    \put(0,0){\includegraphics[width=\unitlength,page=2]{loops.pdf}}%
    \put(0.68503488,0.22060474){\color[rgb]{0,0,0}\makebox(0,0)[lt]{\lineheight{1.25}\smash{\begin{tabular}[t]{l}$-L$\end{tabular}}}}%
    \put(0.69456491,0.18753417){\color[rgb]{0,0,0}\makebox(0,0)[lt]{\lineheight{1.25}\smash{\begin{tabular}[t]{l}$L$\end{tabular}}}}%
    \put(0.15483279,0.0799076){\color[rgb]{0,0,0}\makebox(0,0)[lt]{\lineheight{1.25}\smash{\begin{tabular}[t]{l}Tadpole graph\end{tabular}}}}%
    \put(0.5870997,0.07809215){\color[rgb]{0,0,0}\makebox(0,0)[lt]{\lineheight{1.25}\smash{\begin{tabular}[t]{l}Looping-edge graph\end{tabular}}}}%
    \put(0,0){\includegraphics[width=\unitlength,page=3]{loops.pdf}}%
  \end{picture}%
\endgroup%
}
    \caption{Tadpole and looping-edge-type graphs}
    \label{Fig1}
\end{figure}

A wave function $\mathbf{U}$ defined on $\mathcal{G}$ will be understood 
as an $(N+1)$-tuple of functions $\mathbf{U} = (\phi, (\psi_j)_{j=1}^N) 
= (\phi, \psi_1, \dots, \psi_N)$, where $\phi$ is defined on 
$e_0 = [-L, L]$ and $\psi_j$ is defined on $e_j = [L, \infty)$ for 
$j = 1, \dots, N$.

\medskip
The primary aim of this manuscript is to develop a self-contained and 
systematic theory of linear dynamics for both the Airy and the 
Schr\"odinger operators on looping-edge graphs, intended to serve as a 
reference for future nonlinear applications. We emphasize that this 
abstract theory is presented as a standalone contribution: the dynamical 
part is established here independently, so that subsequent studies on 
existence, stability, controllability, or stabilization of nonlinear 
waves on these graphs can build upon it without having to reconstruct 
the linear theory from scratch.\newline
For the Airy operator, to the best of our knowledge, no results on the 
characterization of linear dynamics on looping-edge graphs appear in 
the existing literature. We fill this gap by providing, in terms of 
Krein space techniques, a complete description of all boundary conditions 
at the vertex that make the Airy operator the generator of a unitary 
group or a contraction semigroup. For the Schr\"odinger operator, our 
approach offers an additional structural advantage compared to the classical methods: it allows one to 
prescribe a priori the type of vertex interactions desired (for 
instance, continuity of derivatives at the vertex) and then to 
systematically identify all self-adjoint extensions compatible with 
those conditions. In this way the theory does not merely produce a 
parameter family of extensions, but organizes them according to 
physically or geometrically motivated constraints.

\medskip

Consider the case of nonlinear Schr\"odinger-type equations (NLS)
\begin{equation}\label{NLS}
i\mathbf{U}_t + \Delta\mathbf{U} + (p+1)|\mathbf{U}|^{2p}\mathbf{U} 
= \mathbf{0}, \quad p>0,
\end{equation}
where the action of the Laplacian operator $\Delta$ on a general graph 
$\mathbb{G}$ is given by
\begin{equation}
-\Delta: (u_e)_{e\in E} \to (-u''_e)_{e\in E}.
\end{equation}
The NLS in \eqref{NLS} has been extensively studied in recent literature 
for various types of metric graphs $\mathbb{G}$ and with specific domains 
for the Laplacian, which make it a self-adjoint operator on 
$L^2(\mathbb{G})$ (see the review manuscript \cite{KNP}). For instance, 
research has focused on star graphs (\cite{AdaNoj14, AdaNoj15, AngGol17a, 
AngGol17b, KPG, Noj14} and references therein), looping-edge graphs 
(\cite{An1, An2, CFN, NP, NPS} and references therein), as well as 
flower graphs, dumbbell graphs, double-bridge graphs, and periodic ring 
graphs (\cite{CFN, BMP, KMPX, KP, MP, NPS, Pan} and references therein).

One of the objectives of this work is to obtain a characterization of 
all self-adjoint extensions of the symmetric operator 
$\mathcal H_0\equiv -\Delta$ on the metric graph $\mathcal{G}$ with 
domain $D(\mathcal H_0)$ defined by
\begin{equation}
    \label{cinf} 
    D(\mathcal H_0)=C_0^\infty(-L,L)\oplus \bigoplus_{j=1}^N 
    C_0^\infty(L,+\infty).
\end{equation}
Since the deficiency indices of $(\mathcal H_0, D(\mathcal H_0))$ are 
equal to $2+N$, the extension theory of Krein and von Neumann provides 
a $(2+N)^2$-parameter family of self-adjoint extensions 
$(-\Delta_{ext}, D(-\Delta_{ext}))$ of $\mathcal{H}_0$ on 
$L^2(\mathcal{G})$, where the action of $-\Delta_{ext}$ is given by 
$-\Delta_{ext} = -\Delta$. Each member of this family defines a unitary 
dynamic for the linear evolution problem
\begin{equation}
\left\{ \begin{array}{ll}
i\mathbf{U}_t = -\Delta_{ext} \mathbf{U}\\
\mathbf{U}(0)= \mathbf{u}_0\in D(-\Delta_{ext}).
\end{array} \right.
\end{equation}
We point out that this approach, based on boundary systems \cite{Schu2015}, 
demonstrates that every extension operator 
$(-\Delta_{ext},D(-\Delta_{ext}))$ can be characterized by the boundary 
conditions at the vertices $\nu = \pm L$, treating the point $-L$ as a 
terminal vertex through a relatively simple matrix process (see 
Theorem~\ref{thm:SAext} below). At this level, this part of the work can 
be seen as a complement to the works of Exner \textit{et al.} 
\cite{Ex, ExS, ExSere}, which, to our knowledge, has not formally 
appeared in the literature on $\mathcal{G}$. Furthermore, in general, 
the extensions $(-\Delta_{ext}, D(-\Delta_{ext}))$ do not necessarily 
satisfy continuity conditions at the boundary of $[-L, L]$.  In the case of looping-edge graphs, the extensions 
must ensure that the wave functions are continuous at least on the loop 
(i.e., $\phi(L) = \phi(-L)$). For example, in the study of electron 
motion in thin metallic ``wires,'' it is required that the wave function 
remains continuous at the junction $\nu = L$. These boundary conditions 
have a clear physical interpretation, representing the conservation of 
probability flow at the junction (see \cite{ExS}).

This physical requirement can be formally expressed through the following 
family of self-adjoint operator extensions 
$(-\Delta, D_{Z, N})_{Z \in \mathbb{R}}$ for 
$(\mathcal{H}_0, D(\mathcal{H}_0))$ with
\begin{equation}\label{DomainN}
\begin{array}{ll}
   D_{Z, N} =\{\mathbf{U}\in H^2(\mathcal{G}): 
   &\phi(L)=\phi(-L)=\psi_1(L)=\cdots=\psi_N(L)\;\text{and}\\
   &\phi'(L)-\phi'(-L)=\sum_{j=1}^N\psi_j'(L+)+Z \psi_1(L)\},
\end{array}
\end{equation}
where for any $n\geq 0$,
$$
H^n(\mathcal{G})=H^n(-L, L)\oplus \bigoplus_{j=1}^N H^n(L, +\infty).
$$
The boundary conditions in \eqref{DomainN} are called $\delta$-type if 
$Z \neq 0$, and Neumann--Kirchhoff type if $Z = 0$. The parameter $Z$ 
is a coupling constant between the loop and the several half-lines. The choice of 
coupling at the vertex $\nu = L$ corresponds to a feasible quantum-wire 
experiment (see \cite{Ex,ExSere} and references therein). It should be 
noted that, based on the domain $D_{Z, N}$, several recent studies have 
been conducted, particularly on the existence and stability of 
standing-wave solutions for the model in \eqref{NLS} (see 
\cite{AST, ASTcri, ASTmul, An1, An2, Ardila, CFN, CFN2, KNP, NP, NPS} 
and references therein). 

A distinctive feature of the approach developed here is that 
it allows one to impose vertex conditions \emph{a priori} and then 
identify all extensions compatible with them, rather than parametrizing 
all extensions at once and then selecting the relevant subfamily. Concrete illustrations are provided in Examples~\ref{ex:8} and \ref{ex:10} below.  The 
existence and (in)stability of standing wave profiles for the NLS model \eqref{NLS} on looping-edge and $\mathcal{T}$-shaped graphs under dynamics induced by some of these extensions is an active research direction pursued by the authors. As a further illustration of this framework, in Section~\ref{sec:application} we 
discuss the orbital instability of elementary (yet novel) standing waves for the cubic NLS on $\mathcal{G}$ under pure periodic plus Neumann type of interactions.

Another important goal of this work is to obtain a characterization of 
all boundary conditions under which the Airy-type evolution equation
\begin{equation}\label{kdv0}
\partial_{t}u_\mathbf{e}(x,t)=\alpha_\mathbf{e} \partial_x^3 
u_\mathbf{e}(x,t) + \beta_\mathbf{e} \partial_x u_\mathbf{e}(x,t), 
\quad x\neq 0,\ t\in \mathbb{R},\ \mathbf{e}\in E,
\end{equation}
with real coefficients given by the sequences 
$(\alpha_\mathbf{e})_{e\in E}$, $(\beta_\mathbf{e})_{e\in E}$, 
generates either a contraction semigroup or a unitary group on a 
looping-edge graph (see Theorems~\ref{prop2.2}, \ref{char} and 
\ref{T1}). Our approach is based on the abstract results in 
\cite{Schu2015} applied to the case of the Airy operator
\begin{equation}
    \label{0airy}
    A_0:\{u_e\}_{e\in E}\mapsto \{\alpha_e \partial_x^3 u_e
    +\beta_e \partial_x u_e\}_{e\in E},
\end{equation}
when we consider the metric graph $\mathcal{G}$ represented by $E$ 
defined above and virtually treating the point $-L$ as a terminal vertex. In contrast with the symmetric operator case $\mathcal{H}_0$ the skew-symmetry of $A_0$ induce a slightly different boundary system and therefore different structure of the characterizations (compare Theorems~\ref{thm:SAext} and \ref{char} below)
Therefore, notions of self-orthogonal subspaces and operators based on 
indefinite inner product spaces (Krein spaces) will be necessary. Thus, 
we obtain in particular a description of all (skew-)self-adjoint 
realizations of $A_0$, i.e., skew-self-adjoint restrictions of $A_0^*$ 
and therefore self-adjoint extensions of $iA_0$. For instance, in 
Example~\ref{ex1} below we obtain $\delta$-interaction type boundary 
conditions on a looping-edge graph in a similar way as done in the case 
of metric star-shaped graphs (see, for instance, Theorem~3.6 in 
\cite{AC}) under conditions on $\alpha_\mathbf{e}$ and $\beta_\mathbf{e}$. 
We also establish conditions for obtaining extensions of $A_0$ 
generating contractive dynamics.

We recall that the evolution linear model in \eqref{kdv0} is associated 
to the following vectorial Korteweg--de Vries (KdV) equation
\begin{equation}\label{kdv3}
\partial_{t}u_\mathbf{e}(x,t)=\alpha_\mathbf{e} \partial_x^3 
u_\mathbf{e}(x,t) + \beta_\mathbf{e} \partial_x u_\mathbf{e}(x,t)
+2 u_\mathbf{e}(x,t) \partial_x u_\mathbf{e}(x,t)
\end{equation}
on the metric graph $\mathcal{G}$. Recently, Mugnolo, Noja, and Seifert 
\cite{MNS} characterized all boundary conditions under which the Airy 
equation in \eqref{kdv0} generates either a semigroup or a group on a 
metric graph that has a structure represented by the set 
$E \equiv E_- \cup E_{+}$, where ${E}_{+}$ and ${E}_{-}$ are finite or 
countable collections of semi-infinite edges $e$ parametrized by 
$(-\infty, 0)$ and $(0, +\infty)$, respectively, connected at a unique 
vertex $\nu = 0$. In this case the graph is sometimes called a 
star-shaped metric graph. Several arguments used in \cite{Schu2015} and 
\cite{MNS} in star-shaped graphs are extended in this work for 
looping-edge graphs. We recall that when the star-shaped metric graph is 
balanced ($|{E}_{+}| = |{E}_{-}|$), the extension theory implies that 
there is a $9|{E}_{+}|^2$-parameter family of skew-self-adjoint 
operators for $A_0$. On a looping-edge graph, when the number of 
attached half-lines is $N=2k$ for some integer $k$ with the conditions 
$\alpha_j,\beta_j>0$ if $j$ is odd and $\alpha_j,\beta_j<0$ if $j$ is 
even, there must be a $9(k+1)^2$-parameter family of unitary groups, 
each one having its own representation.

As far as we know, the study of the KdV equation on metric graphs remains 
relatively underdeveloped. Cavalcante \cite{Cav1} studied the local 
well-posedness of the Cauchy problem associated with \eqref{kdv3} in 
Sobolev spaces $H^s(\mathcal{Y})$ with low regularity on a 
$\mathcal{Y}$-junction graph, 
$\mathcal{Y} = (-\infty, 0) \cup (0, +\infty) \cup (0, +\infty)$. 
Subsequently, Angulo and Cavalcante \cite{AC} established a criterion 
for the linear instability of stationary solutions for the KdV on 
arbitrary star-shaped metric graphs. This criterion was applied to tail 
and bump profiles on balanced graphs with specific boundary conditions 
at the vertex $\nu = 0$, ensuring they remain linearly unstable (see 
Theorem~6.2 in \cite{AC}). More recently, Mugnolo, Noja, and Seifert 
\cite{MNS2} investigated the existence of solitary waves on a 
star-shaped metric graph in terms of the coefficients of the KdV 
equation $(\alpha_\mathbf{e})_{e\in E}$, $(\beta_\mathbf{e})_{e \in E}$ 
on each edge, with coupling conditions at the vertex $\nu = 0$, which 
may or may not involve continuity, and examined the speeds of the 
traveling waves. Lastly, Angulo and Cavalcante \cite{AC2} developed a 
local well-posedness theory for the KdV on a balanced star graph, 
considering boundary conditions at the vertex $\nu = 0$ of 
$\delta$-interaction type, and extended the linear instability result 
from \cite{AC} to establish nonlinear instability.

Results on the dynamics of the KdV model on looping-edge graphs, such 
as local well-posedness of the Cauchy problem and existence and stability 
of solitary waves, are not known in the literature to the best of the 
authors' knowledge. The theory developed in this manuscript is intended 
to fill that gap at the linear level and to lay the groundwork for 
future nonlinear investigations. Several applications are already 
emerging from the framework presented here. On the one hand, the 
boundary system based on separation of first derivatives, developed in 
Section~4.3, has already proven effective in the analysis of dispersive 
equations on metric graphs: the technique was used in \cite{MNS2} to 
construct explicit solitary-wave solutions for the KdV equation on 
star graphs, where the first-derivative coupling at the vertex plays a 
distinguished role. Moreover, the domains generated via separation of 
derivatives in Example~\ref{control} directly recover the looping-edge analogue to the vertex conditions employed in 
\cite{AmCrep,PCP}, works on the stabilization and 
controllability of the KdV on networks with bounded and unbounded domains. In general these abstract theory can be used to yield the class of boundary conditions needed for controllability 
problems on metric graphs, in fact the class of extensions $A_{Y,R_1}$ constitute 
the functional-analytic backbone of an ongoing program on exponential 
stabilization and controllability of the KdV equation on $\mathcal{G}$, 
currently in an advanced stage of research. As a concrete illustration of the use of contractive dynamics, in 
Section~\ref{s4.4} we show that adding a damping term $-\gamma\mathbf{U}$ 
to the Airy evolution under the boundary conditions of 
Example~\ref{ex6} yields exponential decay of the $L^2(\mathcal{G})$-energy 
at rate $\gamma$, with the dissipation arising entirely from the 
boundary structure rather than from the dispersive coefficients.

\vskip0.1in
\noindent\textbf{Organization of the article.} 
Section~2 collects the necessary preliminaries, including notation, 
the Stone and Lummer--Phillips theorems, and the relevant theory of 
operators on Krein spaces. Section~3 analyzes the free Airy operator 
and computes its deficiency indices on the looping-edge graph. 
Section~4 is the core of the paper and is devoted to the Airy operator 
on looping-edge graphs. It is divided into three parts: 
Section~4.1 treats the case of an even number of half-lines and 
characterizes extensions generating unitary dynamics; 
Section~4.2 handles an arbitrary number of half-lines and 
characterizes extensions generating contractive dynamics; 
Section~4.3 develops an alternative boundary system based on 
the separation of derivatives, providing a further class of 
extensions with contractive dynamics. 
Finally, Section~5 applies the same abstract framework to the 
Schr\"odinger operator on looping-edge graphs, characterizing 
its self-adjoint extensions and the associated unitary dynamics.
\section{Preliminaries}
\subsection{Notation}Let us introduce some notation related to the boundary values of a function posed in $\mathcal{G}$. For $\textbf{U}=(\phi,\psi_1,\dots,\psi_N)$, we denote the vector of values of $\textbf{U}$ and its derivatives at the vertex with \begin{equation*}
\begin{split}
    \ &\vec{\mathbf{U}} := \vec{\mathbf{U}}(L) 
= (\phi(-L),\ \phi(L),\ \psi_1(L),\ \dots, \psi_N(L))^T,\\ \ &
\vec{\mathbf{U}}' := \vec{\mathbf{U}}'(L):=(\phi'(-L),\ \phi'(L),\ \psi_1'(L),\ \dots, \psi_N'(L) )^T,\\\ &
\vec{\mathbf{U}}'' := \vec{\mathbf{U}}''(L):=(\phi''(-L),\ \phi''(L),\ \psi_1''(L),\ \dots, \psi_N''(L) )^T.
\end{split}
\end{equation*}

For the components of $\textbf{U}$, we denote the derivatives values at the vertex with $$\partial \phi(-L):=\begin{pmatrix}
    \phi(-L)\\\phi'(-L)\\\phi''(-L)
\end{pmatrix},\ \partial \phi(L):=\begin{pmatrix}
    \phi(L)\\\phi'(L)\\\phi''(L)
\end{pmatrix}
\mbox{ and } \partial \psi_j (L)=\begin{pmatrix}
    \psi_j(L)\\\psi_j'(L)\\\psi_j''(L)
\end{pmatrix}. $$

We adopt concatenation of vectors in vertical stack. For instance, \begin{equation*}
    \begin{pmatrix}
        \partial \phi(-L)\\\partial \phi(L)
    \end{pmatrix}=\begin{pmatrix}
        \phi(-L)\\\phi'(-L)\\\phi''(-L)\\\phi(L)\\\phi'(L)\\\phi''(L)
    \end{pmatrix}.
\end{equation*} 
For convenience we denote for $\psi=(\psi_1,\dots,\psi_N)$
\begin{equation}
    \partial \psi (L):=\begin{pmatrix}
       \partial \psi_1(L)\\ \partial \psi_2(L)\\\vdots\\\partial \psi_N(L)
    \end{pmatrix}, \ \partial \psi_{even} (L):=\begin{pmatrix}
       \partial \psi_2(L)\\ \partial \psi_4(L)\\\vdots\\\partial \psi_{2k}(L)
    \end{pmatrix} \mbox{ and }\partial \psi_{odd} (L):=\begin{pmatrix}
       \partial \psi_1(L)\\ \partial \psi_3(L)\\\vdots\\\partial \psi_{2k-1}(L)
    \end{pmatrix}.
\end{equation}
\subsection{Stone and Lumer-Phillips Theorems}\label{lumm}
In this subsections we recall some well-known features of group theory.

Let $X$ be a Banack space and $X'$ its dual space. Consider $\mathcal{J}(x):=\{x'\in X' \mid \langle x, x'\rangle =\|x\|^2=\|x'\|^2 \}$. Note $\mathcal{J}(x)$ is non-empty by Hahn-Banach Theorem. Moreover, if $X$ is a Hilbert space then $\mathcal{J}(x)$ consists of a single element.

Recall that a linear operator $(A,D(A))$ on a Banach space $X$ is called dissipative if $\|(\lambda-A)x\|\ge \lambda \|x\|$. 

\begin{proposition}
    \label{3.23} An operator $(A,D(A))$ is dissipative if and only if for every $x\in D(A)$ there exists $j(x)\in \mathcal{J}(x)$ such that \begin{equation}\label{A1}
        \mbox{Re}\langle Ax,j(x)\rangle\le0.
    \end{equation}
    If $A$ is the generator of a strongly continuous contraction semigroup, then \eqref{A1} holds for all $x\in D(A)$ and arbitrary $x'\in \mathcal{J}(x)$.
\end{proposition}
\begin{proof}
    See Proposition 3.23 in \cite{Engel}
\end{proof}
Note in the case $X$ is a Hilbert space, condition \eqref{A1} reduces to $\mbox{Re}(Ax\mid x)\le 0$. 
\begin{theorem}(Lummer-Phillips)
    \label{3.17} Let $(A,D(A))$ be a densely defined operator on a Banach space $X$. If both, $A$ and $A^*$ are dissipative, then the closure $\overline{A}$ of $A$ generates a contraction semigroup on $X$. 
\end{theorem}
\begin{proof}
    See Corollary 3.17 in \cite{Engel}.
\end{proof}

We also recall the Stone's theorem for unitary groups.
    \begin{theorem}[Stone]\label{stone}
        Let $(A, D(A))$ be a densely defined operator on a Hilbert space $H$. Then $A$ generates a unitary group $(T(t))_{t\in \R}$ on $H$ if and only if $A$ is skew-self-adjoint, \textit{i.e.}, $A^*=-A$ and $D(A^*)=D(A)$.
    \end{theorem}

\subsection{Operators on Krein Spaces}
Following \cite{Schu2015} we introduce the abstract notions of self-orthogonal subspaces and operators based on indefinite inner product spaces. 
\begin{definition}\label{orto}
   Let $Y$ be a linear space, let $X\subset Y$ be a subspace and $w$ a sesquilinear form defined on $(Y\times Y)^2$. The space $X$ is called $w$-self-orthogonal if \begin{equation}
        X=X^{\perp_w}\equiv \{(x,y)\in Y\times Y\mid w((x,y),(\tilde{x},\tilde{y}))=0 \mbox{ for all } (\tilde{x},\tilde{y})\in X \}.
    \end{equation}
\end{definition}

\begin{definition}
A Krein space $(\mathcal{K},(\cdot\mid\cdot)_{\mathcal{K}})$ is a complex linear space $\mathcal{K}$ on which a indefinite inner product $(\cdot\mid\cdot)_{\mathcal{K}}$ is defined, namely, the function $(x, y)\in \mathcal{K}\times \mathcal{K} \to (x \mid y)_{\mathcal{K}}\in \mathbb C$  satisfies the following axioms:
\begin{enumerate}
\item[i)] Linearity in the first argument:
$$
(\alpha x_1+\beta x_2\mid y)_{\mathcal{K}}= \alpha (x_1\mid y)_{\mathcal{K}}+ \beta (x_2\mid y)_{\mathcal{K}}
$$
for all $x_1,x_2, y\in \mathcal{K}$ and all complex numbers $\alpha, \beta$;
\item[ii)] antisymmetry:
$$
(x \mid y)_{\mathcal{K}}=\overline{(y\mid x)_{\mathcal{K}}}
$$
for all $x,  y\in \mathcal{K}$;
\item[iii)] nondegeneracy: if $(x \mid y)_{\mathcal{K}}=0$ for all $y\in \mathcal{K}$, then $x=0$.
\end{enumerate}
\end{definition}

Thus, the function $(\cdot\mid \cdot)_{\mathcal{K}}$ satisfies all the properties of a standard inner product
with the possible exception that $(x\mid x)_{\mathcal{K}}$ may be nonpositive for $x\neq 0$.

We recall that if $dim(\mathcal K)=n$ ($\mathcal{K}\cong \mathbb C^n$)  every indefinite inner
product  $(\cdot\mid \cdot)_{\mathcal{K}}$ on $\mathcal{K}$ is determined by  an $n \times n$ invertible and hermitian matrix $H $ such that $(x\mid  y)_{\mathcal{K}}=\langle Hx, y\rangle_{\mathbb C^n}$ holds for all $x, y\in \mathcal{K}$.

\begin{definition}
    Let $\mathcal{K}_1, \mathcal{K}_2$ be Krein spaces and a linear operator $\mathcal{L}:\mathcal{K}_1\to \mathcal{K}_2$.
    \begin{enumerate}
        \item $\mathcal{L}$ is called a $(\mathcal{K}_1,\mathcal{K}_2)$-contraction if \begin{equation}\label{E215}
        \left( \mathcal{L}x \mid \mathcal{L}x\right)_{\mathcal{K}_2}\le \left( x \mid x \right)_{\mathcal{K}_1} \mbox{ for all } x\in D(\mathcal{L}).\end{equation}
        \item If $\mathcal{L}$ is densely defined, the $(\mathcal{K}_2,\mathcal{K}_1)$-adjoint of $\mathcal{L}$ is defined by \begin{equation}
        D(\mathcal{L}^\#):=\{y\in \mathcal{K}_2 \mid \exists z\in \mathcal{K}_1: \left( \mathcal{L}x\mid y\right)_{\mathcal{K}_2}=\left( x\mid z\right)_{\mathcal{K}_1} \mbox{ for all }x\in D(\mathcal{L}) \}\mbox{ and }\mathcal{L}^\#y=z.\end{equation}
        \item $\mathcal{L}$ is called $(\mathcal{K}_1,\mathcal{K}_2)$-unitary if $D(\mathcal{L})$ and $R(\mathcal{L})$ are dense, $\mathcal{L}$ is injective and $\mathcal{L}^\#=\mathcal{L}^{-1}$. 
    \end{enumerate}

\end{definition}

\section{The free Airy operator and deficiency indices} \label{defairy}

In this section we fix notation and specify the functional spaces in which the Airy operator will be studied. We also provide an analysis on how the deficiency indices for the free Airy operator behave.

We consider the Airy operator
 \begin{equation}
    \label{airy}
    A_0:\{u_e\}_{e\in E}\mapsto \{\alpha_e \partial_x^3u_e+\beta_e \partial_x u_e\}_{e\in E},
\end{equation}
where $\alpha_j$ and $\beta_j$ are real numbers for any $j$. For $\mathcal{G}$ determined by the edges $e_0=[-L,L]$ and $e_j=[L,\infty)$ considering $-L$ as a {\it terminal vertex}, we understand the action of $A_0$ over the edges in a component-wise manner, that is, for $\textbf{U}=(\phi, \psi_1,\dots,\psi_N)$ posed on $\mathcal{G}$ we denote 
\begin{equation}
    A_0 \textbf{U}=\left(\alpha_0\partial_x^3\phi+\beta_0\partial_x\phi,\  \alpha_1\partial_x^3\psi_1+\beta_1\partial_x\psi_1, \ \dots\ , \ \alpha_N\partial_x^3\psi_N+\beta_N\partial_x\psi_N\right).
\end{equation} We consider $A_0$ as a densely defined operator in \begin{equation}
    L^2(\mathcal{G}):=L^2(-L,L)\oplus \bigoplus_{j=1}^N  L^2(L,\infty)
\end{equation}
with domain \begin{equation}\label{C0}
    D(A_0)=C_0^\infty(-L,L)\oplus \bigoplus_{j=1}^N C_0^\infty(L,\infty).
\end{equation}

For $\textbf{U}=(\phi,\psi_1\dots,\psi_N)$ and $\textbf{V}=(\tilde{\phi},\tilde{\psi}_1,\dots,\tilde{\psi}_N)$, we denote the inner product on $L^2(\mathcal{G})$ with 
\begin{equation}\label{inner}
    [\textbf{U},\textbf{V}]:=[(\phi,\psi_1,\dots,\psi_N),(\tilde{\phi},\tilde{\psi}_1,\dots,\tilde{\psi}_N)]_{L^2(\mathcal{G})}=\int_{-L}^L \phi \overline{\tilde{\phi}} dx + \sum_{j=1}^N\int_L^\infty \psi_j \overline{\tilde{\psi}_j} dx.
\end{equation}
We consider analogously the Sobolev spaces $$H^s(\mathcal{G}):=H^s(-L,L)\oplus \bigoplus_{j=1}^N H^s(L,\infty).$$ 

Note in particular $D(A_0^*)=H^3(\mathcal{G})$, where {\it $A_0^*=-A_0$ is the adjoint of the minimal operator defined as the closure of  $A_0$ with domain $D(A_0)$ in \eqref{C0}}.

For the densely defined symmetric operator $iA_0$ we consider the deficiency spaces \begin{equation}
    \mathcal{D}_{\mp}(iA_0^*):=\mbox{Ker}(iA_0^*\pm iI),
\end{equation}
with deficiency indices $d_{\mp}=\mbox{dim}\mathcal{D}_{\mp}(iA_0^*)$. 

    From the definition of $A_0$ in \eqref{airy} we have that solutions of $-i\alpha_e \partial_x^3u_e-i\beta_e \partial_x u_e\pm iu_e=0$ without any boundary condition are obtained as linear combinations of complex exponentials of the form $$u_e=c_1 e^{r_1x_e}+c_2 e^{r_2x_e}+c_3 e^{r_3x_e}$$ where $r_1, r_2$ and $r_3$ are the roots of $p(x)=\alpha_e x^3+\beta_ex\mp1$. We refer the reader to the proof of Lemma 2.1 in \cite{MNS} for an explicit description. 

We note that $e^{rx}$ belong to $L^2(-L,L)$ independent on $r$, which means $d_{-_0}=d_{+_0}=3$ along $e_0$. On the other hand, on a half-line $[L,\infty)$ we note $e^{rx}$ is square integrable only if the real part of $r$ is negative, which depends on the signs of $\alpha_e$. We summarize the behavior (on a single edge $[L,\infty)$) in the following table.

\begin{center}
\begin{tabular}{ |c|c|c| } 
 \hline
 Sign of $\alpha_e$ & $d_{-_e}$ & $d_{+_e}$\\ \hline
 $\alpha_e>0$ & 2 & 1\\ 
 $\alpha_e<0$ & 1 & 2\\
 \hline
\end{tabular}
\vspace{3mm}
\end{center}

From the Von Neumann and Krein theory, one can expect self-adjoint extensions of $iA_0$ only if the deficiency indices agree, that is, $d_+=d_-$. In such case, according to Theorem \ref{stone}, there exist extensions that generate unitary dynamics. Several combinations of signs of $\alpha_e$ and $\beta_e$ along the edges of $\mathcal{G}$ may lead to equal deficiency indices. For instance, in case the number of attached half-lines is $N=2k$ for some integer $k$ with the conditions $\alpha_j,\beta_j>0$ if $j$ is odd and $\alpha_j,\beta_j<0$ if $j$ is even, we would get deficiency indices $d_+$ and $d_-$ equal to $3k+3$. Thus, in this case, we know that there must be a $9(k+1)^2$ family of unitary groups for the case of the looping-edge graph, each one having its own representation.

\section{The Airy operator on a looping-edge graph}
As mentioned in the previous section, unitary dynamics are expected only in particular cases where the deficiency indices match for the free Airy operator. In  the following subsections, we will introduce different approaches on a looping-edge graph $\mathcal{G}$ based on whether the deficiency indices coincide or not. We start by considering a "pair-wise" reference frame on $\mathcal{G}$ that generates unitary dynamics, (see Theorem \ref{char} below).

\subsection{Boundary system for even amount of half-lines}  Let us first consider $\mathcal{G}$ consisting of an even amount $2k$ of half-lines attached to the vertex in which exactly half of them have a leading coefficient $\alpha_e>0$. More precisely, we consider $\mathcal{G}$ identified trough $V=\{L,-L\}$ and $E=\{[-L,L]\}\cup E_{+}\cup E_{-}$ where $|E_-|=|E_+|=k\in \Z^+$ and $e\in E_{\pm}$ if and only if $\pm\alpha_e>0$.

Our objective in this section is to provide a first description of the boundary values of $A_0^*=-A_0$ ($D(A_0^*)=H^3(\mathcal{G})$), we proceed as in Section 2 of \cite{Schu2015}. For an alternative description see also Lemma \ref{LN} in Section \ref{S2.4}. 

\begin{definition}
  We define in the graph of $A_0^*=-A_0$, 
  $$
  G(A_0^*)=\{(\textbf{U},A_0^* \textbf{U})\mid \textbf{U} \in D(A_0^*)\}
  $$
   the standard anti-symmetric form $\Omega: G(A_0^*)\times G(A_0^*) \to \mathbb{C}, $ defined by \begin{equation}\label{omega}
    \Omega((\textbf{U},A_0^*\textbf{U}),(\textbf{V},A_0^*\textbf{V
})):=[A_0^*\textbf{U}, \textbf{V
}]+[\textbf{U}, A_0^*\textbf{V
}], 
\end{equation}  
with $[\cdot,\cdot]$ being the $L^2(\mathcal G)$-inner product in \eqref{inner}.
\end{definition}

\begin{proposition}
    \label{L1}
    For any $\textbf{U}=(\phi,\psi_1\dots,\psi_N)$ and $\textbf{V}=(\tilde{\phi},\tilde{\psi}_1,\dots,\tilde{\psi}_N)$ in $D(A_0^*)$ we have
    \begin{equation}\label{boundaryf}
        \begin{split}
            &[A_0^*\textbf{U},\textbf{V}]+[\textbf{U},A_0^*\textbf{V}]=\\&\hspace{5mm}\left( \begin{pmatrix} \beta_0 & 0 & \alpha_0 \\ 0 & -\alpha_0 & 0 \\ \alpha_0 & 0 & 0 \end{pmatrix} \begin{pmatrix}
                \phi(-L)\\\phi'(-L)\\\phi''(-L)
            \end{pmatrix}\mid \begin{pmatrix}
                \tilde{\phi}(-L)\\\tilde{\phi}'(-L)\\\tilde{\phi}''(-L)
            \end{pmatrix} \right)  -
            \left( \begin{pmatrix} \beta_0 & 0 & \alpha_0 \\ 0 & -\alpha_0 & 0 \\ \alpha_0 & 0 & 0 \end{pmatrix} \begin{pmatrix}
                \phi(L)\\\phi'(L)\\\phi''(L)
            \end{pmatrix}\mid \begin{pmatrix}
                \tilde{\phi}(L)\\\tilde{\phi}'(L)\\\tilde{\phi}''(L)
            \end{pmatrix} \right)\\&\hspace{5mm}+\sum_{j=1}^N \left( \begin{pmatrix} \beta_j & 0 & \alpha_j \\ 0 & -\alpha_j & 0 \\ \alpha_j & 0 & 0 \end{pmatrix}\begin{pmatrix}
                \psi_j(L)\\\psi_j'(L)\\\psi_j''(L)
            \end{pmatrix} \mid \begin{pmatrix}\tilde{\psi}_j(L)\\\tilde{\psi}_j'(L)\\\tilde{\psi}_j''(L)
            \end{pmatrix}\right)
\\&\hspace{5mm}=\left( B_0\partial \phi(-L)\mid\partial \tilde{\phi}(-L)\right)-\left( B_0\partial \phi(L)\mid\partial \tilde{\phi}(L)\right) +\sum_{j=1}^N\left( B_j \partial \psi_j(L)\mid \partial \tilde{\psi}_j(L)\right). 
        \end{split}
    \end{equation}
where \begin{equation}
   B_j:=\begin{pmatrix} \beta_j & 0 & \alpha_j \\ 0 & -\alpha_j & 0 \\ \alpha_j & 0 & 0 \end{pmatrix},\ j=0,\dots,N,
\end{equation}
and $ (\cdot\mid\cdot)$ represents the standard inner product in $\mathbb C^3$.
\end{proposition}

\begin{proof}
    The proof follows by integration by parts.
\end{proof}

Next,  we define the boundary-subspaces $\mathcal{B}_{\pm}$ as
\begin{equation}\label{B+}
\mathcal{B}_{+}=\Big\{\begin{pmatrix}
                \partial \phi(-L) \\ \partial \psi_{odd}(L)
            \end{pmatrix}: (\phi,\psi_1\dots,\psi_N)\in H^3(\mathcal G)\Big\}
\end{equation}
and 
\begin{equation}\label{B-}
\mathcal{B}_{-}=\Big\{\begin{pmatrix}
                \partial \phi(L) \\ \partial \psi_{even}(L)
            \end{pmatrix}: (\phi,\psi_1\dots,\psi_N)\in H^3(\mathcal G)\Big\}.
\end{equation}
            
In this form, we have the following definition.

\begin{definition}\label{def22}
    Consider the spaces $\mathcal{B}_{+}\cong\mathcal{B}_{-}\cong \mathbb{C}^{3(k+1)}$ in \eqref{B+}-\eqref{B-} and endow these spaces with the non-degenerated indefinite inner products \begin{equation}
        \left( x\mid y\right)_{\mathcal{B}_+}:=\left( x\mid y\right)_{+}=\left( B_+ x\mid y\right)\mbox{ and  }\ \left( x\mid y\right)_{\mathcal{B}_-}:=\left( x\mid y\right)_-:=\left( B_- x\mid y\right),
    \end{equation}
    where $ (\cdot\mid\cdot)$ represents the standard inner product in $\mathbb C^{3(k+1)}$ and 
    
    \begin{equation*}
        B_+:=\operatorname{diag}(B_0,B_1,B_3,\dots,B_{2k-1})\ \mbox{ and  }\ 
            B_-:=\operatorname{diag}(B_0,-B_2,-B_4,\dots,-B_{2k}).
    \end{equation*}
   $ (\mathcal{B}_{\pm}, \left( \cdot\mid \cdot\right)_{\pm})$  are Krein spaces.

\end{definition}

Under these new notations, Proposition \ref{L1} can be rewritten in a compact form as 
\begin{equation}\label{E225}
    [A_0^*\textbf{U},\textbf{V}]+[\textbf{U},A_0^*\textbf{V}]=\left( \begin{pmatrix}
        \partial \phi(-L)\\\partial \psi_{odd}(L)
    \end{pmatrix}\mid \begin{pmatrix}
        \partial \tilde{\phi}(-L)\\\partial \tilde{\psi}_{odd}(L)
    \end{pmatrix} \right)_+ - \left( \begin{pmatrix}
        \partial \phi(L)\\\partial \psi_{even}(L)
    \end{pmatrix}\mid \begin{pmatrix}
        \partial \tilde{\phi}(L)\\\partial \tilde{\psi}_{even}(L)
    \end{pmatrix} \right)_-.
\end{equation}
\begin{definition}\label{w-sesqui}
   We define the function
    \begin{equation}
        \begin{split}
            F:& \ G(A_0^*)\to \mathcal{B}_{+}\oplus\mathcal{B}_{-}\\ &(\textbf{U},A_0^*\textbf{U})\mapsto \left(\begin{pmatrix}
                \partial \phi(-L) \\ \partial \psi_{odd}(L)
            \end{pmatrix},  \begin{pmatrix}
                \partial \phi(L) \\\partial \psi_{even}(L)
            \end{pmatrix}\right)
        \end{split}
    \end{equation}
    and the sesquilinear form 
    \begin{equation}
        \begin{split}
            w:&\ (\mathcal{B}_{+}\oplus \mathcal{B}_{-})\times (\mathcal{B}_{+}\oplus \mathcal{B}_{-})\to \C\\ &((x,y),(\tilde{x}, \tilde{y}))\mapsto \left( x\mid \tilde{x}\right)_{+}-\left( y\mid \tilde{y}\right)_{-}.
        \end{split}
    \end{equation}
\end{definition}

Note from \eqref{omega} and \eqref{boundaryf} that
\begin{equation}\label{E1}
    \Omega((\textbf{U},A_0^*\textbf{U}),(\textbf{V},A_0^*\textbf{V}))=w(F(\textbf{U},A_0^*\textbf{U}),F(\textbf{V},A_0^*\textbf{V})).
\end{equation}

By \cite{Schu2015} we can now characterize skew-self-adjoint extensions $ A$ of
$A_0$ - {\it i.e.}, skew-self-adjoint restrictions of $A_0^*$ Ð and therefore self-adjoint extensions of $iA_0$. Note that obviously $G(A)=\{(x, Ax): x\in D(A)\}\subset G(A_0^*)$

\begin{proposition}\label{prop2.2}
    An extension $A$ of $A_0$ is skew-self-adjoint if and only if there exists a $w$-self-orthogonal subspace $X\subset \mathcal{B}_{+}\oplus\mathcal{B}_{-}$ such that $G(A)=F^{-1}(X)$.
\end{proposition}

\begin{proof}
It follows immediately from Definitions \ref{orto} and \ref{w-sesqui}, relations \eqref{E225} and \eqref{E1}, and from \cite{Schu2015}.
\end{proof}

Hence, $w$-self-orthogonal subspaces $X$ parametrize the skew-self-adjoint extensions $A$ of $A_0$. A more explicit
description of these objects is given next.

\subsubsection{Extensions generating unitary dynamics}

\begin{theorem}
    \label{char}Let $X\subset \mathcal{B}_+\oplus\mathcal{B}_-$. $X$ is $w$-self-orthogonal if and only if there is a $(\mathcal{B}_+,\mathcal{B}_-)$-unitary operator $\mathcal{L}$ such that $X=G(\mathcal{L})$. 
\end{theorem}
\begin{proof}
We follow the same ideas of \cite{MNS}. For sake of completeness we carry on with the details.  Let $X$ be $w$-self-orthogonal and let $(x,y)\in X$. By definition we have $w((x,y),(x,y))=0$, that is, \begin{equation}
    \label{eq1}(x\mid x)_{+}=(y\mid y)_{-}.
\end{equation} Since $X$ is a subspace, by setting $x=0$ we get $0=(y\mid y)_-$. Since $(\cdot\mid\cdot)_-$ is non-degenerated it must be the case $y=0$. By linearity we can conclude that $X$ is the graph of a linear operator $\mathcal{L}:\mathcal{B}_+\to\mathcal{B}_-$. Also note that if $y=0$ it also implies $x=0$; which means $\mathcal{L}$ is injective. 

To see $\mathcal{L}$ is densely defined, let $z\in D(\mathcal{L})^\perp$ and $x\in D(\mathcal{L})$. We have
\begin{equation}
    0=(z\mid x)=(B_+B_+^{-1}z\mid x)=(B_+^{-1}z\mid x)_+=w((B_+^{-1}z, 0), (x,\mathcal{L}x) ).
\end{equation}

Since $x\in D(\mathcal{L})$ is arbitrary, we conclude $(B_+^{-1}z,0)\in X^{\perp_w}=X=G(\mathcal{L})$. Since $\mathcal{L}$ is injective, it must be the case $B^{-1}_+z=0$. The latter implies $z=0$. A symmetric argument shows the image $R(\mathcal{L})$ is dense. 

Let us now see $\mathcal{L}^\#=L^{-1}$. First note for $x\in D(\mathcal{L})$ and $y\in R(\mathcal{L})$, from \eqref{eq1}, we have $$(\mathcal{L}x\mid y)_-=(x\mid \mathcal{L}^{-1}y)_+.$$ The latter implies $R(\mathcal{L})\subset D(\mathcal{L}^\#)$ with $\mathcal{L}^\#y=\mathcal{L}^{-1}y$; \textit{i.e.}, $\mathcal{L}^{-1}\subset \mathcal{L}^\#$. On the other hand, if $y\in D(\mathcal{L}^\#)$ with $\mathcal{L}^\#y=x$ we have 
$$(\mathcal{L}z \mid y)_{-}=(z \mid x )_+\ \ \mbox{for all }z\in D(\mathcal{L}).$$
Note the latter is equivalent to say $$w((z,\mathcal{L}z),(x,y))=0\ \mbox{ for all }z\in D(\mathcal{L}),$$
which implies $(x,y)\in X^{\perp_w}=X=G(\mathcal{L})$. Thus, $(y,x)\in G(\mathcal{L}^{-1})$. We conclude $\mathcal{L}^{-1}=\mathcal{L}^\#$; proving $\mathcal{L}$ is unitary.

Conversely, let $\mathcal{L}$ be a $(\mathcal{B}_+,\mathcal{B}_-)$-unitary operator with $X=G(\mathcal{L})$. Let $x\in D(\mathcal{L})$. For any $z\in D(\mathcal{L})$ we have $(\mathcal{L}x\mid\mathcal{L}z)_-=(x\mid z)_+$. That is, $w((x,\mathcal{L}x),(z,\mathcal{L}z))=0$ for any $z\in D(\mathcal{L})$. We conclude $(x,\mathcal{L}x)\in X^{\perp_w}$, \textit{i.e.}, $X\subset X^{\perp_w}$. On the other hand, if $(x,y)\in X^{\perp_w}$ we have $w((x,y),(z,\mathcal{L}z))=0$ for all $z\in D(\mathcal{L})$, \textit{i.e.}, $(\mathcal{L}z\mid y)_-=(z\mid x)_+$ for all $z\in D(\mathcal{L})$. By definition of $\mathcal{L}^\#$ the latter implies $y\in D(\mathcal{L}^\#)=D(\mathcal{L}^{-1})$ with $\mathcal{L}^{-1}y=\mathcal{L}^\#y=x$. We conclude $(x,y)\in G(\mathcal{L})=X$. Hence $X=X^{\perp_w}$ as desired. 
\end{proof}
\subsubsection{Examples}

With the first example below we intend to resemble the so-called $\delta$ \textit{interactions} on a looping-edge graph $\mathcal{G}$ in a similar way as done in the case of metric star-shaped graphs (see, for instance, Theorem 3.6 in \cite{AC}). 

\begin{example}[$\delta$-type interactions on a looping-edge graph]\label{ex1}
    Let us first consider $k=1$, \textit{i.e.}, two half-lines attached to a common vertex. Assume $-\alpha_2=\alpha_1=\alpha_0$ and $-\beta_2=\beta_1=\beta_0$. In particular,
    \begin{equation}
    A_0 =\left(\alpha_0\partial_x^3+\beta_0\partial_x,\  \alpha_0\partial_x^3+\beta_0\partial_x,  \  -\alpha_0\partial_x^3-\beta_0\partial_x\right),
\end{equation} 
 and $B_+=B_-$ in Definition \ref{def22}.

For parameters $z,m\in \R$ consider the operator
\begin{equation}
\mathcal{L}:=\begin{pmatrix}
        1&0&0&0&0&0\\
        z&1&0&0&0&0\\
        \frac{z^2}{2}&z&1&m&0&0\\
        0&0&0&1&0&0\\
        0&0&0&z&1&0\\
        -m&0&0&\frac{z^2}{2}&z&1
    \end{pmatrix}.
\end{equation}
Since $\mathcal{L}^\#=B_+^{-1}\mathcal{L}^*B_-=\mathcal{L}^{-1}$, it can be seen that $\mathcal{L}$ defines a $(\mathcal{B}_+,\mathcal{B}_-)$-unitary operator. According to Proposition \ref{prop2.2} and Theorem \ref{char}, $\mathcal{L}$ defines a skew-self-adjoint extension $A_\mathcal{L}$ of $A_0$ defined via the relation $F(G(A_\mathcal{L}))=G(\mathcal{L})$. That is, an extension with domain $D(A_\mathcal{L})$ consisting of the functions $\textbf{V}=(\phi, \psi_1, \psi_{2})\in D(A_0^*)$ such that
\begin{equation}\label{ddom}
    \begin{split}
       & \phi(L)=\phi(-L),\ \psi_1(L)=\psi_2(L),\ \phi'(L)-\phi'(-L)=z\phi(-L),\ \\& \psi'_2(L)-\psi'_1(L)=z\psi_1(L),\ \phi''(L)-\phi''(-L)=\frac{z^2}{2}\phi(-L)+z\phi'(-L)+m\psi_1(L)\\& \mbox{and }\psi_2''(L)-\psi_1''(L)=\frac{z^2}{2}\psi_1(L)+z\psi_1'(L)-m\phi(-L).
    \end{split}
\end{equation}
Note in particular that if $\textbf{V}\in D(A_{\mathcal{L}})$ is so that $\phi(-L)=\psi_1(L)$ we may conclude from \eqref{ddom} that \begin{equation}
    \label{ddomcont}
    \begin{split}
        &\phi(-L)=\phi(L)=\psi_1(L)=\psi_2(L), \\& \left(\phi'(L)+\psi'_2(L) \right)-\left(\phi'(-L)+\psi'_1(L)\right)=2z\phi(-L).
    \end{split}
\end{equation}

    

More generally, consider $k\in \Z^+$ arbitrary (\textit{i.e.} $2k$ half-lines attached to the common vertex $\nu=L$) and a parameter $z\in \R$. Under the analogous assumption $-\alpha_{2j}=\alpha_{2j-1}=\alpha_0$, $-\beta_{2j}=\beta_{2j-1}=\beta_0$ for all $j\in\{1,\dots,k\}$, so we obtain in particular for $k=2$ 
 \begin{equation}
    A_0 =\left(\alpha_0\partial_x^3+\beta_0\partial_x,\  \alpha_0\partial_x^3+\beta_0\partial_x,  \  -\alpha_0\partial_x^3-\beta_0\partial_x,\ \alpha_0\partial_x^3+\beta_0\partial_x,\  - \alpha_0\partial_x^3-\beta_0\partial_x \right).
\end{equation} 
Define the following operator with diagonal matrix of blocks 
\begin{equation}
    \mathcal{L_{\delta_Z}}:=\operatorname{diag}(\delta_z,\delta_z,\dots,\delta_z)\ \mbox{ where }\ \delta_z:=\begin{pmatrix}
        1&0&0\\ z&1&0\\ \frac{z^2}{2}&z&1
    \end{pmatrix}.
\end{equation}
Note that each block $\delta_{z}$ is such that $B_0^{-1}\delta_z^*B_0=\delta_z^{-1}$ for any $z\in \R$ ($B_0$ in \eqref{boundaryf}) and since $\mathcal{L}_{\delta_z}^\#$ is a diagonal matrix of blocks of the form $B_0^{-1}\delta_z^*B_0$ it can be seen that $\mathcal{L}_{\delta_z}^\#=\mathcal{L}_{\delta_z}^{-1}$. From there it is easy to see that $\mathcal{L}_{\delta_z}$ is a $(\mathcal{B}_+,\mathcal{B}_-)$-unitary operator. From Proposition \ref{prop2.2} and Theorem \ref{char}, $\mathcal{L}_{\delta_z}$  defines a skew-self-adjoint extension $A_{\mathcal{L}_{\delta_z}}$ of $A_0$ with domain 
\begin{equation}\label{set1}
    D_{\delta_z}:=\left\{ \textbf{V}\in H^3(\mathcal{G})\ \mid \delta_z \partial\phi(-L)=\partial \phi(L) \mbox{ and } \delta_z \partial \psi_{2j-1}(L)=\partial\psi_{2j}(L),\ j=1,\dots,k\right\}.
\end{equation}
Note from \eqref{set1} we obtained "pair-wise" the conditions in \eqref{ddom}. More precisely, abusing notation to identify $\psi_{-1}$ with $\phi$ at $-L$ and $\psi_{0}$ with $\phi$ at $L$ we have for any $j=0,\dots,k$ that
\begin{equation}
\begin{split}
     &\psi_{2j}(L)=\psi_{2j-1}(L),\ \psi_{2j}'(L)-\psi_{2j-1}'(L)=z\psi_{2j}(L) \mbox{ and}\\& \psi_{2j}''(L)-\psi_{2j-1}''(L)=z\psi_{2j-1}'(L)+\frac{z^2}{2}\psi_{2j-1}(L).
\end{split}
\end{equation}

Moreover, if in particular $\textbf{V}\in D_{\delta_z}$ is such that $\phi(-L)=\psi_{2j-1}(L)$ for all $j=1,\dots, k$ then we conclude from \eqref{set1} that 
\begin{equation}
    \phi(-L)=\phi(L)=\psi_1(L)=\cdots=\psi_{2k}(L) \mbox{ and } \sum_{j=0}^k \psi'_{2j-1}(L) - \sum_{j=0}^k\psi'_{2j}(L)=(1+k)z\phi(-L).
\end{equation}   
\end{example}

For the following examples we assume $|E_+|=|E_-|=1$. From the deficiency indices in Section \ref{defairy} and the Von Neumann and Krein theory we know the skew-self-adjoint extensions of $A_0$ are parametrized by $6^2=36$ parameters. We  exhibit a $2$ and a $4$-parameter independent families of unitary operators that maintain the looping-edge structure (\textit{i.e.} the loop condition $\phi(-L)=\phi(L)$) and are not necessarily interacting pair-wise.  

\begin{example}
    Consider again $k=1$ and $B_+=B_-$.
It can be seen for $m_1,m_2\in \R$ that \begin{equation}
 \mathcal{L}=\begin{pmatrix}
         1&0&0&0&0&0\\
           0&0&0&m_1&1&0\\
           0&0&1&m_2&0&0\\
           0&0&0&1&0&0\\
           0&1&0&0&0&0\\
          -m_2&0&0&\frac{m_1^2}{2}&m_1&1
       \end{pmatrix}
   \end{equation} is a unitary operator and therefore that it defines a skew-self-adjoint extension $A_\mathcal{L}$ with domain $D(A_\mathcal{L})$ made of the functions $v=(\phi, \psi_1, \psi_{2})\in D(A_0^*)$ such that
\begin{equation*}
\begin{split}
&\phi(-L)=\phi(L),\  \psi_1(L)=\psi_2(L),\  \phi'(-L)=\psi'_2(L),\ \phi'(L)-\psi'_1(L)=m_1\psi_1(L),\\& \phi''(L)-\phi''(-L)=m_2\psi_1(L)\ \mbox{and }\psi_2''(L)-\psi_2''(L)=\frac{m_1^2}{2}\psi_1(L)+m_1\psi_1'(L)-m_2\phi(-L).
\end{split}
\end{equation*}
\end{example}

\begin{example}
Consider again $k=1$, $-\alpha_2=\alpha_1=\alpha_0$ and $-\beta_2=\beta_1=\beta_0$. It can be seen for $m_1,\dots,m_4\in \R$ that 
\begin{equation}
    \mathcal{L}=\begin{pmatrix}
        1&0&0&0&0&0\\
        m_1&0&0&m_2&1&0\\
        \frac{m_1^2+1}{2}&1&1&m_3&m_1&0\\
        0&0&0&1&0&0\\
        1&1&0&m_4&0&0\\
        m_4-m_3+m_1m_2&m_4&0&\frac{m_2^2+m_4^2}{2}&m_2&1
    \end{pmatrix}
\end{equation}
defines a skew-self-adjoint extension $A_\mathcal{L}$ with domain $D(A_\mathcal{L})$ made of $v\in D(A_0^*)$ such that \begin{equation*}
    \begin{split}
        &\phi(L)=\phi(-L),\ \psi_1(L)=\psi_2(L),\ \phi'(L)-\psi_1'(L)=m_1\phi(-L)+m_2\psi_1(L), \\& \psi_2'(L)-\phi'(-L)=m_4\psi_1(L)+\phi(-L),\\&        \phi''(L)-\phi''(-L)=\frac{m_1^2+1}{2}\phi(-L)+m_3\psi_1(L)+\phi'(-L)+m_1\psi_1'(L), \mbox{and}\\& \psi_2''(L)-\psi_1''(L)=(m_4-m_3+m_1m_2)\phi(-L)+\frac{m_2^2+m_4^2}{2}\psi_1(L)+m_4\phi'(-L)+m_2\psi_1'(L). 
    \end{split}
\end{equation*}
\end{example}
\subsection{Boundary system for arbitrary amount of half-lines}

For this part we no longer seek unitary dynamics. We allow the dynamics to be
generated by a $C_0$-semigroup of contractions. We consider $\mathcal{G}$
consisting of a finite amount $N$ of half-lines attached to the vertex $L$,
with no restrictions on the signs of the coefficients of the Hamiltonian. We
still virtually treat $-L$ as a terminal vertex.

\begin{definition}
Consider the spaces $\mathcal{C}_1\cong\mathcal{C}_0\cong\mathbb{C}^{3(N+1)}$
equipped with the non-degenerate indefinite inner products
\begin{equation}
    (x\mid y)_1:=(C_1\,x\mid y),\qquad (x\mid y)_0:=(C_0\,x\mid y),
\end{equation}
where
\begin{equation*}
    C_1:=\operatorname{diag}(B_0,B_1,B_2,\dots,B_N),\qquad
    C_0:=\frac{1}{N+1}\operatorname{diag}(B_0,B_0,\dots,B_0).
\end{equation*}
Both $(\mathcal{C}_1,(\cdot\mid\cdot)_1)$ and $(\mathcal{C}_0,(\cdot\mid\cdot)_0)$
are Krein spaces.
\end{definition}

\subsubsection{Boundary trace maps and the Green identity}

We introduce boundary trace maps $\Gamma_1, \Gamma_0: D(A_0^*)\to\mathbb{C}^{3(N+1)}$ by
\begin{equation}\label{eq:traces}
    \Gamma_1\mathbf{U}:=
    \begin{pmatrix}
        \partial\phi(-L)\\[2pt]\partial\psi_1(L)\\\vdots\\\partial\psi_N(L)
    \end{pmatrix}\in\mathcal{C}_1,
    \qquad
    \Gamma_0\mathbf{U}:=
    \begin{pmatrix}
        \partial\phi(L)\\\partial\phi(L)\\\vdots\\\partial\phi(L)
    \end{pmatrix}\in\mathcal{C}_0.
\end{equation}
In terms of these maps, the Green identity for $A_0^*$ reads
\begin{equation}\label{eq:GI}
    [A_0^*\mathbf{U},\mathbf{V}]+[\mathbf{U},A_0^*\mathbf{V}]
    =(\Gamma_1\mathbf{U}\mid\Gamma_1\mathbf{V})_1 - (\Gamma_0\mathbf{U}\mid\Gamma_0\mathbf{V})_0.
\end{equation}

\subsubsection{Boundary relations, Krein adjoints, and contractive extensions}

\begin{definition}\label{def:BR}
A \emph{boundary relation} is a linear subspace $R\subset\mathcal{C}_1\times\mathcal{C}_0$.
Its \emph{Krein adjoint} $R^{[*]}\subset\mathcal{C}_0\times\mathcal{C}_1$ is defined by
\begin{equation}\label{eq:Kadj}
    (y,x)\in R^{[*]}
    \;\Longleftrightarrow\;
    (y_1\mid y)_0=(x_1\mid x)_1
    \quad\text{for all }(x_1,y_1)\in R.
\end{equation}
The relation $R$ is called a \emph{$(\mathcal{C}_1,\mathcal{C}_0)$-contraction} if
\begin{equation}\label{eq:Rcontr}
    (y\mid y)_0\le(x\mid x)_1
    \qquad\text{for all }(x,y)\in R.
\end{equation}
When $R=\operatorname{graph}(\mathcal{L})$ for a linear map
$\mathcal{L}:\mathcal{C}_1\to\mathcal{C}_0$, the Krein adjoint satisfies
$R^{[*]}=\operatorname{graph}(\mathcal{L}^\#)$ with
\begin{equation}\label{eq:Lsharp}
    \mathcal{L}^\#:=C_1^{-1}\mathcal{L}^*C_0,
\end{equation}
and contractivity of $R$ reduces to $\mathcal{L}^*C_0\mathcal{L}-C_1\preceq 0$.
\end{definition}

\begin{remark}\label{rem:Dirichlet}
The \emph{multivalued part} of $R$ is
$\operatorname{mul}(R):=\{y:(0,y)\in R\}$.
Whenever $\operatorname{mul}(R)\ne\{0\}$, the relation $R$ is not the graph of
any operator, and no matrix $\mathcal{L}$ can encode the corresponding
boundary conditions. Concretely, $(0,y)\in R$ imposes
$\Gamma_1\mathbf{U}=0$, i.e.
\[
    \phi(-L)=\phi'(-L)=\phi''(-L)=0,\quad
    \psi_j(L)=\psi_j'(L)=\psi_j''(L)=0,\quad j=1,\dots,N,
\]
while leaving $\Gamma_0\mathbf{U}$ unconstrained. These are
\emph{Dirichlet-type} conditions at the incoming edges; they are valid
generators of contraction semigroups whenever the corresponding relation is
contractive, but are inaccessible to the operator-based framework.
\end{remark}

\begin{definition}\label{def:AR}
Given a boundary relation $R\subset\mathcal{C}_1\times\mathcal{C}_0$, define
\begin{equation}\label{eq:AR}
    A_R\mathbf{U}:=-A_0^*\mathbf{U},\qquad
    D(A_R):=\bigl\{\mathbf{U}\in D(A_0^*):(\Gamma_1\mathbf{U},\Gamma_0\mathbf{U})\in R\bigr\}.
\end{equation}
\end{definition}

Let us describe the adjoint of $A_R$.
\begin{proposition}\label{L3}
$A_R^*=-A_{R^{[*]}}$, where $A_{R^{[*]}}$ acts as $A_0^*$ with domain
\begin{equation}
    D\!\left(A_{R^{[*]}}\right)
    :=\bigl\{\mathbf{V}\in D(A_0^*):(\Gamma_0\mathbf{V},\Gamma_1\mathbf{V})\in R^{[*]}\bigr\}.
\end{equation}
\end{proposition}

\begin{proof}
Let $\mathbf{U}\in D(A_R)$ and $\mathbf{V}\in D(A_0^*)$.
Using \eqref{eq:GI} and $A_R\mathbf{U}=-A_0^*\mathbf{U}$,
\begin{equation}
    [A_R\mathbf{U},\mathbf{V}]
    =[\mathbf{U},A_0^*\mathbf{V}]
    +(\Gamma_0\mathbf{U}\mid\Gamma_0\mathbf{V})_0
    -(\Gamma_1\mathbf{U}\mid\Gamma_1\mathbf{V})_1.
\end{equation}
Hence $\mathbf{V}\in D(A_R^*)$ if and only if
\[
    (y_1\mid\Gamma_0\mathbf{V})_0=(x_1\mid\Gamma_1\mathbf{V})_1
    \quad\forall\,(x_1,y_1)\in R,
\]
which by \eqref{eq:Kadj} is equivalent to
$(\Gamma_0\mathbf{V},\Gamma_1\mathbf{V})\in R^{[*]}$.
In this case $A_R^*\mathbf{V}=A_0^*\mathbf{V}=-(-A_0^*\mathbf{V})=-A_{R^{[*]}}\mathbf{V}$.
\end{proof}

We now obtain dissipativity via contractivity in the following sense:
\begin{proposition}\label{P1}
$A_R$ is dissipative if and only if $R$ is a $(\mathcal{C}_1,\mathcal{C}_0)$-contraction.
\end{proposition}

\begin{proof}
For $\mathbf{U}\in D(A_R)$, set $(x,y):=(\Gamma_1\mathbf{U},\Gamma_0\mathbf{U})\in R$.
The Green identity \eqref{eq:GI} applied to the pair $(\mathbf{U},\mathbf{U})$ gives
\begin{equation}
    2\operatorname{Re}[A_R\mathbf{U},\mathbf{U}]
    =(y\mid y)_0-(x\mid x)_1.
\end{equation}
Hence $\operatorname{Re}[A_R\mathbf{U},\mathbf{U}]\le0$ for all $\mathbf{U}\in D(A_R)$
if and only if $(y\mid y)_0\le(x\mid x)_1$ for all $(x,y)\in R$,
which is exactly \eqref{eq:Rcontr}.
\end{proof}

\begin{remark}\label{rem:auto} We note contractivity of $R$ implies contractivity of $R^{[*]}$ in this finite-
dimensional setting. In fact, by the polar decomposition of an invertible Hermitian matrix we can write $C_0=J_0H_0$ and $C_1=J_1H_1$ with $H_0,H_1$ positive definite and
$J_0,J_1$ fundamental symmetries of the respective Krein spaces. The
contractivity condition \eqref{eq:Rcontr} for $R$, when expressed through the
similarity transforms $H_0^{1/2}$ and $H_1^{1/2}$, becomes a Hilbert-space
contraction condition for an equivalent relation. Its adjoint in the Hilbert
sense is again a contraction, and undoing the similarity transforms shows that
$R^{[*]}$ satisfies \eqref{eq:Rcontr} with the roles of $\mathcal{C}_0$ and
$\mathcal{C}_1$ interchanged. Consequently the
dissipativity of $A_R$ automatically entails the dissipativity of
$A_R^*=-A_{R^{[*]}}$, without any additional hypothesis on $R$.
\end{remark}

\begin{theorem}\label{T1}
Let $R\subset\mathcal{C}_1\times\mathcal{C}_0$ be a linear boundary relation.
Then $A_R$ generates a $C_0$-semigroup of contractions on $L^2(\mathcal{G})$
if and only if $R$ is a $(\mathcal{C}_1,\mathcal{C}_0)$-contraction.
\end{theorem}

\begin{proof}
Suppose $R$ is contractive. By Proposition~\ref{P1}, $A_R$ is dissipative.
By Remark~\ref{rem:auto}, $R^{[*]}$ is also contractive, so Proposition~\ref{P1}
applied to $R^{[*]}$ gives that $A_R^*=-A_{R^{[*]}}$ is dissipative. The
Lumer--Phillips theorem then yields that $A_R$ generates a contraction
semigroup.

Conversely, if $A_R$ generates a contraction semigroup, the Lumer--Phillips
theorem implies $A_R$ is dissipative, and Proposition~\ref{P1} gives that $R$
is contractive.
\end{proof}

\subsubsection{Examples}

We now present extensions $A_R$ generating semigroups of contractions, framed
as boundary relations. The first example corresponds to an operator-type
relation (graph of a matrix), while Remark~\ref{rem:Dirichlet} describes the
additional Dirichlet-type extensions that fall outside the operator framework.

\begin{example}[$\delta$-type interactions on a tadpole graph]\label{ex4}
Let $N=1$, $\alpha_0=\alpha_1=\beta_0=\beta_1$, and let $m_1,\dots,m_4\in\mathbb{R}$.
Define the matrix
\begin{equation}\label{eq:Mmat}
    M:=\begin{pmatrix}
        1 & 0 & 0 & 0 & 0 & 0\\[2pt]
        \dfrac{m_2+m_3}{2} & 1 & 0 & 1 & 0 & 0\\[6pt]
        m_1 & m_2+m_3 & 2 & m_2 & 0 & 0\\[2pt]
        0 & 0 & 0 & 1 & 0 & 0\\[2pt]
        \dfrac{m_2+m_3}{2} & 1 & 0 & 1 & 0 & 0\\[6pt]
        m_3 & 2 & 0 & m_4 & 0 & 2
    \end{pmatrix},
\end{equation}
and consider the boundary relation
$R:=\operatorname{graph}(M)\subset\mathcal{C}_1\times\mathcal{C}_0$,
i.e.\
\begin{equation}
    R = \bigl\{(x,Mx): x\in\mathcal{C}_1\bigr\}.
\end{equation}
By Definition~\ref{def:BR}, $R$ is a $(\mathcal{C}_1,\mathcal{C}_0)$-contraction
if and only if the matrix $M^*C_0 M-C_1$ is negative semidefinite. A direct
computation shows that a sufficient condition is
\begin{equation}\label{hyp}
    \alpha_0<0,\qquad m_4\le\tfrac{3}{2},\qquad
    4m_1-(m_2^2+m_3^2)\le2.
\end{equation}
By Theorem~\ref{T1}, under \eqref{hyp} the operator $A_R$ generates a
$C_0$-semigroup of contractions.

The Krein adjoint relation is $R^{[*]}=\operatorname{graph}(M^\#)$ with
$M^\#=C_1^{-1}M^*C_0$, given explicitly by
\begin{equation}
    M^\#=\begin{pmatrix}
        1 & 0 & 0 & 0 & 0 & 0\\[2pt]
        \dfrac{-m_2-m_3}{2} & \dfrac{1}{2} & 0 & -1 & \dfrac{1}{2} & 0\\[8pt]
        \dfrac{m_1-1}{2} & \dfrac{-(m_2+m_3)}{4} & \dfrac{1}{2} & \dfrac{m_3}{2} & \dfrac{-(m_2+m_3)}{4} & 0\\[8pt]
        0 & 0 & 0 & 1 & 0 & 0\\[2pt]
        0 & 0 & 0 & 0 & 0 & 0\\[2pt]
        \dfrac{m_2}{2} & -\dfrac{1}{2} & 0 & \dfrac{m_4-1}{2} & -\dfrac{1}{2} & \dfrac{1}{2}
    \end{pmatrix},
\end{equation}
which is also a $(\mathcal{C}_0,\mathcal{C}_1)$-contraction under \eqref{hyp},
in agreement with Remark~\ref{rem:auto}.

Explicitly, the domain of $A_R$ consists of all $(\phi,\psi)\in D(A_0^*)$
satisfying $(\Gamma_1(\phi,\psi),\,\Gamma_0(\phi,\psi))\in R$, that is,
\begin{equation}\label{AMz}
    \begin{split}
        &\phi(-L)=\phi(L)=\psi(L),\\[4pt]
        &\phi'(L)-\phi'(-L)
        =\left(\tfrac{m_2+m_3+2}{2}\right)\phi(-L),\\[4pt]
        &\phi''(L)-2\psi''(L)
        =(m_3+m_4)\,\psi(L)+2\,\phi'(-L),\\[4pt]
        &\phi''(L)-2\phi''(-L)
        =(m_1+m_2)\,\phi(-L)+(m_2+m_3)\,\phi'(-L).
    \end{split}
\end{equation}
In particular, setting $m_4=m_1=0$, $m_3=-2$, and writing $z=m_2/2\in\mathbb{R}$,
the conditions \eqref{AMz} reduce to
\begin{equation}
    \begin{split}
        &\phi(-L)=\phi(L)=\psi(L),\qquad
        \phi'(L)-\phi'(-L)=z\,\phi(-L),\\[4pt]
        &\phi''(L)-\phi''(-L)-\psi''(L)
        =z\,\phi'(-L)+(z-1)\,\phi(-L).
    \end{split}
\end{equation}

The adjoint operator $A_R^*=-A_{R^{[*]}}$ has domain consisting of all
$(\phi,\psi)\in D(A_0^*)$ satisfying
$(\Gamma_0(\phi,\psi),\,\Gamma_1(\phi,\psi))\in R^{[*]}$, i.e.
\begin{equation}
    \begin{split}
        &\phi(-L)=\phi(L)=\psi(L),\qquad \psi'(L)=0,\\[4pt]
        &-\tfrac{m_2+m_3}{2}\,\phi(L)+\tfrac{1}{2}\,\phi'(L)-\phi(-L)
        +\tfrac{1}{2}\,\psi'(L) = \phi'(-L),\\[4pt]
        &\tfrac{m_1-1}{2}\,\phi(L)-\tfrac{m_2+m_3}{4}\,\phi'(L)
        +\tfrac{1}{2}\,\phi''(L)+\tfrac{m_3}{2}\,\phi(-L)
        -\tfrac{m_2+m_3}{4}\,\psi'(L) = \phi''(-L),\\[4pt]
        &\tfrac{m_2}{2}\,\phi(L)-\tfrac{1}{2}\,\phi'(L)
        +\tfrac{m_4-1}{2}\,\phi(-L)
        -\tfrac{1}{2}\,\psi'(L)+\tfrac{1}{2}\,\phi''(L) = \psi''(L).
    \end{split}
\end{equation}
\end{example}

\begin{remark} A canonical example of extensions not reachable by linear operators is the
partial Dirichlet relation
\begin{equation}
    R_j:=\bigl\{(x,y)\in\mathcal{C}_1\times\mathcal{C}_0:
    x_j=0\bigr\},
\end{equation}
where $x_j\in\mathbb{C}^3$ is the $j$-th block of $x$, which imposes
$\partial\psi_j(L)=0$ (i.e.\ $\psi_j(L)=\psi_j'(L)=\psi_j''(L)=0$) while
leaving $\Gamma_0\mathbf{U}$ unconstrained. Such relations are contractive
whenever the residual quadratic form $(y\mid y)_0$ is non-positive on
$\{0\}\times\operatorname{mul}(R)$. No finite matrix $\mathcal{L}$ can encode these
boundary conditions, since that would require $\mathcal{L}\cdot 0=y\ne0$,
which is impossible for a linear map.
\end{remark}
\subsection{Boundary system separating derivatives}\label{S2.4}

A finer description of the boundary behavior is obtained by separating
the quadratic form involving first derivatives from those involving the
function values and second derivatives.

The technique of separating the first-derivative component from the 
boundary system, developed in the present subsection, is not merely a 
formal device. It has already proven to be an effective tool in the 
qualitative analysis of dispersive equations on metric graphs: in 
\cite{MNS2}, this separation of derivatives was exploited to construct solitary-wave solutions for the 
Korteweg--de Vries equation on star graphs. Furthermore, the class of 
extensions $A_{Y,R_1}$ introduced below, together with the contraction 
semigroups they generate, constitutes the functional-analytic backbone 
of an ongoing program aimed at the stabilization and controllability of 
the KdV equation on metric graphs.

\medskip
We recall the notation
$\vec{\mathbf{U}}$, $\vec{\mathbf{U}}'$, $\vec{\mathbf{U}}''$
introduced in Subsection~2.1.

\begin{lemma}\label{LN}
For any $\mathbf{U}=(\phi,\psi_1,\dots,\psi_N)$ and
$\mathbf{V}=(\tilde\phi,\tilde\psi_1,\dots,\tilde\psi_N)$ in $D(A_0^*)$,
\begin{equation}\label{boundaryf2}
    [A_0^*\mathbf{U},\mathbf{V}]+[\mathbf{U},A_0^*\mathbf{V}]
    =\bigl(D_\beta\vec{\mathbf{U}}\mid\vec{\mathbf{V}}\bigr)
    +\bigl(D_\alpha\vec{\mathbf{U}}''\mid\vec{\mathbf{V}}\bigr)
    +\bigl(D_\alpha\vec{\mathbf{U}}\mid\vec{\mathbf{V}}''\bigr)
    -\bigl(D_\alpha\vec{\mathbf{U}}'\mid\vec{\mathbf{V}}'\bigr),
\end{equation}
where
\begin{equation}
    D_\alpha:=\operatorname{diag}(\alpha_0,-\alpha_0,\alpha_1,\dots,\alpha_N),
    \qquad
    D_\beta:=\operatorname{diag}(\beta_0,-\beta_0,\beta_1,\dots,\beta_N).
\end{equation}
\end{lemma}

\begin{definition}
Consider the spaces $\mathcal{K}_+\cong\mathcal{K}_-\cong\mathbb{C}^{N+1}$
equipped with the weighted inner products
\begin{equation}
    (x\mid y)_{\mathcal{K}_+}:=(D_\alpha^+\,x\mid y),
    \qquad
    (x\mid y)_{\mathcal{K}_-}:=(D_\alpha^-\,x\mid y),
\end{equation}
where
\begin{equation}
    D_\alpha^+:=\operatorname{diag}(\alpha_0,\alpha_1,\dots,\alpha_N),
    \qquad
    D_\alpha^-:=\frac{1}{N+1}\operatorname{diag}(\alpha_0,\alpha_0,\dots,\alpha_0).
\end{equation}
Both $(\mathcal{K}_\pm,(\cdot\mid\cdot)_{\mathcal{K}_\pm})$ are Krein spaces.
\end{definition}

The matrices $D_\alpha^-$ and $D_\alpha^+$ split $D_\alpha$: the former
collects the coefficient of $\phi(L)$ and the latter collects all remaining
coefficients. Specifically, \eqref{boundaryf2} and the identity
\begin{equation}\label{cuc}
    \bigl(D_\alpha\vec{\mathbf{U}}'\mid\vec{\mathbf{V}}'\bigr)
    =\left(\begin{pmatrix}\phi'(-L)\\\psi_1'(L)\\\vdots\\\psi_N'(L)\end{pmatrix}
    \;\middle|\;
    \begin{pmatrix}\tilde\phi'(-L)\\\tilde\psi_1'(L)\\\vdots\\\tilde\psi_N'(L)\end{pmatrix}
    \right)_{\mathcal{K}_+}
    -
    \left(\begin{pmatrix}\phi'(L)\\\phi'(L)\\\vdots\\\phi'(L)\end{pmatrix}
    \;\middle|\;
    \begin{pmatrix}\tilde\phi'(L)\\\tilde\phi'(L)\\\vdots\\\tilde\phi'(L)\end{pmatrix}
    \right)_{\mathcal{K}_-}
\end{equation}
yield the split Green identity
\begin{equation}\label{eq:GIsplit}
\begin{split}
    [A_0^*\mathbf{U},\mathbf{V}]+[\mathbf{U},A_0^*\mathbf{V}]
    &=\Bigl(D_\alpha\vec{\mathbf{U}}''+\tfrac12 D_\beta\vec{\mathbf{U}}
      \;\Big|\;\vec{\mathbf{V}}\Bigr)
    +\Bigl(\vec{\mathbf{U}}\;\Big|\;\tfrac12 D_\beta\vec{\mathbf{V}}
      +D_\alpha\vec{\mathbf{V}}''\Bigr)\\
    &\quad
    -\left(\begin{pmatrix}\phi'(-L)\\\psi_1'(L)\\\vdots\\\psi_N'(L)\end{pmatrix}
    \;\middle|\;
    \begin{pmatrix}\tilde\phi'(-L)\\\tilde\psi_1'(L)\\\vdots\\\tilde\psi_N'(L)\end{pmatrix}
    \right)_{\mathcal{K}_+}
    +\left(\begin{pmatrix}\phi'(L)\\\phi'(L)\\\vdots\\\phi'(L)\end{pmatrix}
    \;\middle|\;
    \begin{pmatrix}\tilde\phi'(L)\\\tilde\phi'(L)\\\vdots\\\tilde\phi'(L)\end{pmatrix}
    \right)_{\mathcal{K}_-}.
\end{split}
\end{equation}

\subsubsection{Boundary relations for the first-derivative part}

We introduce derivative trace maps
$\Lambda_+,\Lambda_-:D(A_0^*)\to\mathcal{K}_\pm$ by
\begin{equation}\label{eq:Lamps}
    \Lambda_+\mathbf{U}
    :=\begin{pmatrix}\phi'(-L)\\\psi_1'(L)\\\vdots\\\psi_N'(L)\end{pmatrix},
    \qquad
    \Lambda_-\mathbf{U}
    :=\begin{pmatrix}\phi'(L)\\\phi'(L)\\\vdots\\\phi'(L)\end{pmatrix}.
\end{equation}

\begin{definition}\label{def:R1}
A \emph{derivative boundary relation} is a linear subspace
$R_1\subset\mathcal{K}_+\times\mathcal{K}_-$.
Its \emph{Krein adjoint} $R_1^{[*]}\subset\mathcal{K}_-\times\mathcal{K}_+$
is defined by
\begin{equation}\label{eq:R1adj}
    (y,x)\in R_1^{[*]}
    \;\Longleftrightarrow\;
    (y_1\mid y)_{\mathcal{K}_-}=(x_1\mid x)_{\mathcal{K}_+}
    \quad\text{for all }(x_1,y_1)\in R_1.
\end{equation}
The relation $R_1$ is called a $(\mathcal{K}_+,\mathcal{K}_-)$-contraction if
\begin{equation}\label{eq:R1contr}
    (y\mid y)_{\mathcal{K}_-}\le(x\mid x)_{\mathcal{K}_+}
    \qquad\text{for all }(x,y)\in R_1.
\end{equation}
When $R_1=\operatorname{graph}(\mathcal{L})$ for a linear map
$\mathcal{L}:\mathcal{K}_+\to\mathcal{K}_-$, the Krein adjoint satisfies
$R_1^{[*]}=\operatorname{graph}(\mathcal{L}^\#)$ with
$\mathcal{L}^\#:=(D_\alpha^+)^{-1}\mathcal{L}^*D_\alpha^-$,
and \eqref{eq:R1contr} reduces to
$(D_\alpha^+)^{-1}\mathcal{L}^*D_\alpha^-\mathcal{L}-I\preceq0$.
\end{definition}

\subsubsection{Extensions generating contractive dynamics via splitting derivatives}

\begin{definition}\label{def2.8}
Let $R_1\subset\mathcal{K}_+\times\mathcal{K}_-$ be a linear boundary relation
and let $Y\subset\mathbb{C}^{N+2}$ be a closed subspace. Define
\begin{equation}
\begin{split}
    D(A_{Y,R_1})
    :=\Bigl\{\mathbf{U}\in D(A_0^*) \;\Big|\;
    &\vec{\mathbf{U}}\in Y,\quad
    D_\alpha\vec{\mathbf{U}}''+\tfrac12 D_\beta\vec{\mathbf{U}}\in Y^\perp,\\
    &\text{and}\quad
    (\Lambda_+\mathbf{U},\,\Lambda_-\mathbf{U})\in R_1\Bigr\},
\end{split}
\end{equation}
with $A_{Y,R_1}\mathbf{U}:=-A_0^*\mathbf{U}$.
\end{definition}

The subspace condition $\vec{\mathbf{U}}\in Y$ and the orthogonality condition
$D_\alpha\vec{\mathbf{U}}''+\frac12 D_\beta\vec{\mathbf{U}}\in Y^\perp$ together
govern how function values and second derivatives interact at the vertex.
The relation condition $(\Lambda_+\mathbf{U},\Lambda_-\mathbf{U})\in R_1$
governs the first-derivative coupling independently.

\begin{proposition}\label{P3}
Let $R_1\subset\mathcal{K}_+\times\mathcal{K}_-$ be a linear boundary relation
and $Y\subset\mathbb{C}^{N+2}$ a closed subspace. Then
$A_{Y,R_1}^*=-A_{Y,R_1^{[*]}}$, where $A_{Y,R_1^{[*]}}$ acts as $A_0^*$
with domain
\begin{equation}
\begin{split}
    D\!\left(A_{Y,R_1^{[*]}}\right)
    =\Bigl\{\mathbf{V}\in D(A_0^*) \;\Big|\;
    &\vec{\mathbf{V}}\in Y,\quad
    \tfrac12 D_\beta\vec{\mathbf{V}}+D_\alpha\vec{\mathbf{V}}''\in Y^\perp,\\
    &\text{and}\quad
    (\Lambda_-\mathbf{V},\,\Lambda_+\mathbf{V})\in R_1^{[*]}\Bigr\}.
\end{split}
\end{equation}
\end{proposition}

\begin{proof}
Let $\mathbf{U}\in D(A_{Y,R_1})$ and $\mathbf{V}\in D(A_0^*)$.
Using \eqref{eq:GIsplit} and $A_{Y,R_1}\mathbf{U}=-A_0^*\mathbf{U}$,
\begin{equation}\label{E10}
\begin{split}
    [A_{Y,R_1}\mathbf{U},\mathbf{V}]
    &=[\mathbf{U},A_0^*\mathbf{V}]
    -\Bigl(D_\alpha\vec{\mathbf{U}}''+\tfrac12 D_\beta\vec{\mathbf{U}}
     \;\Big|\;\vec{\mathbf{V}}\Bigr)
    -\Bigl(\vec{\mathbf{U}}\;\Big|\;\tfrac12 D_\beta\vec{\mathbf{V}}
     +D_\alpha\vec{\mathbf{V}}''\Bigr)\\
    &\quad
    +(\Lambda_+\mathbf{U}\mid\Lambda_+\mathbf{V})_{\mathcal{K}_+}
    -(\Lambda_-\mathbf{U}\mid\Lambda_-\mathbf{V})_{\mathcal{K}_-}.
\end{split}
\end{equation}
We need $[A_{Y,R_1}\mathbf{U},\mathbf{V}]=[\mathbf{U},A_{Y,R_1}^*\mathbf{V}]$
for all $\mathbf{U}\in D(A_{Y,R_1})$.

Taking $\mathbf{U}$ with $\vec{\mathbf{U}}=\vec{\mathbf{U}}'=0$
gives $(D_\alpha\vec{\mathbf{U}}''\mid\vec{\mathbf{V}})=0$ for all
$D_\alpha\vec{\mathbf{U}}''\in Y^\perp$, hence $\vec{\mathbf{V}}\in Y$.

For arbitrary $x\in Y$, taking $\mathbf{U}$ with
$\vec{\mathbf{U}}=x$ and $\vec{\mathbf{U}}'=0$, the condition
$D_\alpha\vec{\mathbf{U}}''+\frac12 D_\beta\vec{\mathbf{U}}\in Y^\perp$
and \eqref{E10} yield
\[
    \Bigl(x\;\Big|\;\tfrac12 D_\beta\vec{\mathbf{V}}+D_\alpha\vec{\mathbf{V}}''\Bigr)
    =\Bigl(D_\alpha\vec{\mathbf{U}}''+\tfrac12 D_\beta\vec{\mathbf{U}}
     \;\Big|\;\vec{\mathbf{V}}\Bigr)=0,
\]
hence $\frac12 D_\beta\vec{\mathbf{V}}+D_\alpha\vec{\mathbf{V}}''\in Y^\perp$.

At this point \eqref{E10} reduces to the
requirement that for all $\mathbf{U}\in D(A_{Y,R_1})$,
\[
    (\Lambda_+\mathbf{U}\mid\Lambda_+\mathbf{V})_{\mathcal{K}_+}
    =(\Lambda_-\mathbf{U}\mid\Lambda_-\mathbf{V})_{\mathcal{K}_-}.
\]
Since $(\Lambda_+\mathbf{U},\Lambda_-\mathbf{U})\in R_1$, the definition
\eqref{eq:R1adj} of the Krein adjoint shows this is equivalent to
$(\Lambda_-\mathbf{V},\Lambda_+\mathbf{V})\in R_1^{[*]}$.
\end{proof}

We now obtain, once more, dissipativity via contractivity.
\begin{proposition}\label{P4}
$A_{Y,R_1}$ is dissipative if and only if $R_1$ is a
$(\mathcal{K}_+,\mathcal{K}_-)$-contraction.
\end{proposition}

\begin{proof}
For $\mathbf{U}\in D(A_{Y,R_1})$, set $(x,y):=(\Lambda_+\mathbf{U},\Lambda_-\mathbf{U})\in R_1$.
The split Green identity \eqref{eq:GIsplit} applied to $(\mathbf{U},\mathbf{U})$,
together with the conditions $\vec{\mathbf{U}}\in Y$ and
$D_\alpha\vec{\mathbf{U}}''+\frac12 D_\beta\vec{\mathbf{U}}\in Y^\perp$, gives
\begin{equation}
    2\operatorname{Re}[A_{Y,R_1}\mathbf{U},\mathbf{U}]
    =(x\mid x)_{\mathcal{K}_+}-(y\mid y)_{\mathcal{K}_-}.
\end{equation}
Hence $\operatorname{Re}[A_{Y,R_1}\mathbf{U},\mathbf{U}]\le0$ for all
$\mathbf{U}\in D(A_{Y,R_1})$ if and only if $(y\mid y)_{\mathcal{K}_-}\le(x\mid x)_{\mathcal{K}_+}$
for all $(x,y)\in R_1$, i.e.\ $R_1$ is contractive.
\end{proof}

\begin{remark}\label{rem:auto2}
The same Hilbert-space similarity argument as in Remark~\ref{rem:auto}
shows that contractivity of $R_1$ automatically implies
contractivity of $R_1^{[*]}$. Hence dissipativity of $A_{Y,R_1}$ entails
dissipativity of $A_{Y,R_1}^*$.
\end{remark}

\begin{theorem}\label{T2}
Let $R_1\subset\mathcal{K}_+\times\mathcal{K}_-$ be a linear boundary relation
and $Y\subset\mathbb{C}^{N+2}$ a closed subspace. Then $A_{Y,R_1}$ generates
a $C_0$-semigroup of contractions on $L^2(\mathcal{G})$ if and only if $R_1$
is a $(\mathcal{K}_+,\mathcal{K}_-)$-contraction.
\end{theorem}

\begin{proof}
If $R_1$ is contractive, Proposition~\ref{P4} gives dissipativity of
$A_{Y,R_1}$, and Remark~\ref{rem:auto2} gives dissipativity of
$A_{Y,R_1}^*$. The Lumer--Phillips theorem yields the conclusion.
Conversely, if $A_{Y,R_1}$ generates a contraction semigroup it is
dissipative, and Proposition~\ref{P4} gives contractivity of $R_1$.
\end{proof}

\subsubsection{Examples}

For simplicity we take $N=1$ (tadpole graph) and $\alpha_0=\alpha_1$.

\begin{example}\label{ex5}
For $m_1,m_2,m_3\in\mathbb{R}$ with $m_3\ne0$, consider the operator
\[
    \mathcal{L}:=\begin{pmatrix}m_1&m_2\\m_3&-\tfrac{m_1m_2}{m_3}\end{pmatrix}
\]
and the operator-type relation $R_1:=\operatorname{graph}(\mathcal{L})
\subset\mathcal{K}_+\times\mathcal{K}_-$.
Its Krein adjoint is $R_1^{[*]}=\operatorname{graph}(\mathcal{L}^\#)$ with
\[
    \mathcal{L}^\#=\begin{pmatrix}\tfrac{m_1}{2}&\tfrac{m_3}{2}\\[4pt]
    \tfrac{m_2}{2}&-\tfrac{m_1m_2}{2m_3}\end{pmatrix}.
\]
Both $R_1$ and $R_1^{[*]}$ are contractive (hence, by Theorem~\ref{T2},
$A_{Y,R_1}$ generates a contraction semigroup) under the conditions
\begin{equation}\label{cond1}
    m_3^2+m_1^2\le2,\quad
    m_2^2(m_1^2+1)\le2,\quad
    m_1(m_3^2-m_2^2)=0,\quad
    m_1^2+m_2^2\le1,\quad
    \frac{m_1^2m_2^2}{m_3^2}+m_3^2\le2.
\end{equation}
In the case $m_3=m_2\ne0$, \eqref{cond1} reduces to
$m_2^2+m_1^2\le2$ and $m_2^2(m_1^2+1)\le2$.
In particular, $m_1=m_2=1$ gives the contractive relation
$R_1=\operatorname{graph}\!\begin{pmatrix}1&1\\1&-1\end{pmatrix}$.
\end{example}

The next example characterizes extensions that maintain the looping-edge
structure ($\phi(-L)=\phi(L)$) and include continuity at the vertex.

\begin{example}\label{ex6}
Let $z\in\mathbb{R}$ and $R_1\subset\mathcal{K}_+\times\mathcal{K}_-$ be a
contractive boundary relation. Consider the subspace
$Y_z:=\operatorname{span}\{(1,1,z)^\top\}$.

\medskip
\noindent\textbf{Case $z\ne0$.}\quad
$Y_z^\perp=\bigl\{(x,y,w)^\top\in\mathbb{C}^3:w=-(x+y)/z\bigr\}$.
By Definition~\ref{def2.8}, the domain of $A_{Y_z,R_1}$ consists of all
$\mathbf{U}\in D(A_0^*)$ satisfying
\begin{equation}
    \phi(-L)=\phi(L),\quad
    \psi(L)=z\,\phi(L),\quad
    z\!\left(\alpha_1\psi''(L)+\tfrac{\beta_1}{2}\psi(L)\right)
    =\alpha_0\!\left(\phi''(L)-\phi''(-L)\right),
\end{equation}
together with $(\Lambda_+\mathbf{U},\Lambda_-\mathbf{U})\in R_1$.

Taking $z=1$ and $R_1=\operatorname{graph}\!\begin{pmatrix}1&1\\1&-1\end{pmatrix}$
(which is contractive by Example~\ref{ex5}),
the domain specializes to
\begin{equation}
\begin{split}
    D(A_{Y_1,R_1})=\Bigl\{\mathbf{U}\in D(A_0^*)\;\Big|\;
    &\phi(L)=\phi(-L)=\psi(L),\quad\phi'(L)=\phi'(-L),\\
    &\psi'(L)=0,\quad
    \alpha_1\psi''(L)+\tfrac{\beta_1}{2}\psi(L)=\alpha_0\bigl(\phi''(L)-\phi''(-L)\bigr)\Bigr\},
\end{split}
\end{equation}
and by Theorem~\ref{T2}, $A_{Y_1,R_1}$ generates a contraction semigroup on
the tadpole graph.

\medskip
\noindent\textbf{Case $z=0$.}\quad
$Y_0^\perp=\{(x,y,w)^\top\in\mathbb{C}^3:x=-y\}$.
The domain of $A_{Y_0,R_1}$ consists of all $\mathbf{U}\in D(A_0^*)$ satisfying
\begin{equation}
    \phi(-L)=\phi(L),\quad\psi(L)=0,\quad\phi''(L)=\phi''(-L),
\end{equation}
together with $(\Lambda_+\mathbf{U},\Lambda_-\mathbf{U})\in R_1$.

Taking $R_1=\operatorname{graph}\!\begin{pmatrix}1&1\\1&-1\end{pmatrix}$,
the domain specializes to
\begin{equation}\label{hol}
\begin{split}
    D(A_{Y_0,R_1})=\Bigl\{\mathbf{U}\in D(A_0^*)\;\Big|\;
    &\phi(L)=\phi(-L),\quad\psi(L)=0,\\
    &\phi'(L)=\phi'(-L),\quad\phi''(L)=\phi''(-L)\Bigr\},
\end{split}
\end{equation}
and $A_{Y_0,R_1}$ generates a contraction semigroup on the tadpole graph.
\end{example}

\begin{example}\label{control}
Assume $\alpha_j\equiv\alpha$ and $\beta_j\equiv\beta$ for all 
$j=0,\dots,N$. Let
\[
    Y:=\operatorname{span}\{(0,0,1,\dots,1)^T\}\subset\mathbb{R}^{N+2},
\]
so that
\[
    Y^\perp=\left\{(y_1,y_2,x_1,\dots,x_N)\in\mathbb{R}^{N+2}
    \;\Bigm|\;\sum_{j=1}^N x_j=0\right\}.
\]
For $N\in\mathbb{N}$ arbitrary and parameters $m_1,\dots,m_N\in\mathbb{R}$ 
satisfying $m_j^2\le N+1$ for all $j=1,\dots,N$, consider the 
operator-type boundary relation
\begin{equation}\label{operat}
    R_1:=\operatorname{graph}(\mathcal{L}),\qquad
    \mathcal{L}:=\operatorname{diag}(1,m_1,\dots,m_N)
    :\mathcal{K}_+\to\mathcal{K}_-.
\end{equation}
Its Krein adjoint relation is $R_1^{[*]}=\operatorname{graph}(\mathcal{L}^\#)$ 
with $\mathcal{L}^\#=\frac{1}{N+1}\mathcal{L}$. The condition 
$m_j^2\le N+1$ ensures that both $R_1$ and $R_1^{[*]}$ are 
$(\mathcal{K}_+,\mathcal{K}_-)$-contractions in the boundary system 
of Section~\ref{S2.4}, so that Theorem~\ref{T2} applies.\newline
The corresponding extension $A_{Y,R_1}$ is defined on
\begin{equation}\label{Dgen}
\begin{split}
    D(A_{Y,R_1}):=\Big\{\mathbf{U}\in H^3(\mathcal{G})\;\Bigm|\;
    &\phi(-L)=\phi(L)=0,\quad
    \psi_1(L)=\psi_j(L)\text{ for all }j=1,\dots,N,\\
    &\phi'(-L)=\phi'(L)=m_j\psi_j'(L)\text{ for all }j=1,\dots,N,\\
&\sum_{j=1}^N\psi_j''(L)=\frac{-\beta}{2\alpha}\,N\psi_1(L)\Big\},
\end{split}
\end{equation}
and is the infinitesimal generator of a $C_0$-semigroup of contractions 
on $L^2(\mathcal{G})$.

\medskip
The special case $m_j=0$ for all $j=1,\dots,N$ corresponds to the 
boundary relation
\[
R_1^0:=\operatorname{graph}(\operatorname{diag}(1,0,\dots,0)),
\]
whose multivalued part satisfies 
$\operatorname{mul}(R_1^0)=\{(0,y_2,\dots,y_N): y_j\in\mathbb{C}\}$, 
reflecting the fact that the first-derivative traces 
$\psi_1'(L),\dots,\psi_N'(L)$ are forced to zero without any 
constraint on $\Lambda_-\mathbf{U}$. This is a genuine relation, not 
representable by any invertible operator, and the domain reduces to
\begin{equation}\label{dcont}
\begin{split}
    D(A_{Y,R_1^0})=\Big\{\mathbf{U}\in H^3(\mathcal{G})\;\Bigm|\;
    &\phi(L)=\phi(-L)=\phi'(L)=\phi'(-L)=0,\\
    &\psi_1(L)=\psi_j(L)\text{ for all }j=1,\dots,N,\\
    &\psi_j'(L)=0\text{ for all }j=1,\dots,N,\\
    &\sum_{j=1}^N\psi_j''(L)=\frac{-\beta}{2\alpha}\,N\psi_1(L)\Big\}.
\end{split}
\end{equation}
The operator $A_{Y,R_1^0}$ generates a $C_0$-semigroup of contractions 
by Theorem~\ref{T2}, and its domain coincides with the looping-edge analogous of the class of vertex 
conditions employed in the study of stabilization and controllability 
of dispersive equations on star-shaped networks \cite{AmCrep,PCP}.
\end{example}
\subsection{A simple application to exponential stabilization}\label{s4.4}

For simplicity we work within the setting of Example~\ref{ex6} (case $z=0$) on the
tadpole graph $\mathcal{G}$ with $N=1$, taking $\alpha_0=\alpha_1=\alpha>0$ and
$\beta_0=\beta_1=\beta\in\mathbb{R}$. The operator $A=A_{Y_0,R_1}$ from \eqref{hol},
together with the relation
$R_1=\operatorname{graph}\!\begin{pmatrix}1&1\\1&-1\end{pmatrix}$,
acts on all $\mathbf{U}=(\phi,\psi)\in H^3(\mathcal{G})$ satisfying
\[
    \phi(-L)=\phi(L),\qquad
    \psi(L)=0,\qquad
    \phi'(-L)=\phi'(L),\qquad
    \phi''(-L)=\phi''(L),
\]
together with $\psi'(L)=0$. This last condition follows directly from $R_1$:
the equation $\Lambda_-\mathbf{U}=\mathcal{L}\,\Lambda_+\mathbf{U}$ reads
\[
    \begin{pmatrix}\phi'(L)\\\phi'(L)\end{pmatrix}
    =\begin{pmatrix}1&1\\1&-1\end{pmatrix}
    \begin{pmatrix}\phi'(-L)\\\psi'(L)\end{pmatrix},
\]
and subtracting the two rows forces $\psi'(L)=0$, after which either row gives
$\phi'(L)=\phi'(-L)$. Fix $\gamma>0$ and consider the damped problem
\begin{equation}\label{eq:damped-ex6}
    \partial_t\mathbf{U}(t)=A\,\mathbf{U}(t)-\gamma\,\mathbf{U}(t),
    \qquad t>0,\qquad
    \mathbf{U}(0)=\mathbf{U}_0\in D(A).
\end{equation}

\begin{theorem}\label{thm:stab}
Let $\mathbf{U}(t)$ be a classical solution of \eqref{eq:damped-ex6}. Then the
energy $E(t):=\tfrac{1}{2}\|\mathbf{U}(t)\|^2_{L^2(\mathcal{G})}$ satisfies
\[
    E(t)\le e^{-2\gamma t}\,E(0),\qquad t\ge0,
\]
or equivalently $\|\mathbf{U}(t)\|_{L^2(\mathcal{G})}\le
e^{-\gamma t}\|\mathbf{U}_0\|_{L^2(\mathcal{G})}$ for all $t\ge0$.
In particular, the solution decays to zero exponentially at rate $\gamma$,
independently of the dispersive coefficients $\alpha$ and $\beta$.
\end{theorem}

\begin{proof}
We first address the existence of classical solutions to \eqref{eq:damped-ex6}.
Since $A_{Y_0,R_1}$ generates a $C_0$-semigroup of contractions on
$L^2(\mathcal{G})$ by Theorem~\ref{T2}, the operator $A-\gamma I$ is dissipative
for every $\gamma>0$. Moreover, $(A-\gamma I)^*=A^*-\gamma I$, and since $A^*$
is also dissipative (as noted in Remark~\ref{rem:auto2}), so is $A^*-\gamma I$.
The Lumer--Phillips theorem therefore guarantees that $A-\gamma I$ generates a
$C_0$-semigroup of contractions on $L^2(\mathcal{G})$, and in particular that
\eqref{eq:damped-ex6} admits a unique classical solution
$\mathbf{U}\in C^1([0,\infty);L^2(\mathcal{G}))\cap C([0,\infty);D(A))$
for every $\mathbf{U}_0\in D(A)$.

Writing \eqref{eq:damped-ex6} componentwise gives
$\phi_t=-\alpha\phi'''-\beta\phi'-\gamma\phi$ on $(-L,L)$ and
$\psi_t=-\alpha\psi'''-\beta\psi'-\gamma\psi$ on $(L,\infty)$.
Differentiating the energy and substituting the equations of motion yields
\begin{align*}
    \frac{d}{dt}E(t)
    &=-\alpha\!\int_{-L}^{L}\!\phi\phi'''\,dx
     -\beta\!\int_{-L}^{L}\!\phi\phi'\,dx
     -\gamma\!\int_{-L}^{L}\!\phi^2\,dx
     -\alpha\!\int_{L}^{\infty}\!\psi\psi'''\,dx
     -\beta\!\int_{L}^{\infty}\!\psi\psi'\,dx
     -\gamma\!\int_{L}^{\infty}\!\psi^2\,dx.
\end{align*}
We claim that all terms except those with factor $-\gamma$ vanish identically.
Indeed, integrating by parts twice on the loop edge,
\[
    \int_{-L}^{L}\phi\,\phi'''\,dx
    =\Bigl[\phi''\phi\Bigr]_{-L}^{L}-\int_{-L}^{L}\phi''\phi'\,dx
    =\bigl(\phi''(L)\phi(L)-\phi''(-L)\phi(-L)\bigr)
    -\frac{1}{2}\bigl(\phi'(L)^2-\phi'(-L)^2\bigr),
\]
and every boundary term cancels upon applying $\phi(-L)=\phi(L)$,
$\phi'(-L)=\phi'(L)$, and $\phi''(-L)=\phi''(L)$, giving
$\int_{-L}^{L}\phi\,\phi'''\,dx=0$.
Besides, since $\phi\phi'=\tfrac{1}{2}(\phi^2)'$,
\[
    \int_{-L}^{L}\phi\,\phi'\,dx
    =\frac{1}{2}\bigl(\phi(L)^2-\phi(-L)^2\bigr)=0,
\]
again by $\phi(-L)=\phi(L)$. On the other hand, functions in $D(A)$ belong to
$H^3(L,\infty)$, so $\psi$, $\psi'$, and $\psi''$ all decay to zero as
$x\to\infty$. Integrating by parts on the half-line therefore gives
\[
    \int_{L}^{\infty}\psi\,\psi'''\,dx
    =\Bigl[\psi''\psi\Bigr]_{L}^{\infty}-\int_{L}^{\infty}\psi''\psi'\,dx
    =-\psi''(L)\psi(L)+\frac{1}{2}\psi'(L)^2=0,
\]
where the last equality uses $\psi(L)=0$ and $\psi'(L)=0$. Similarly,
$\int_{L}^{\infty}\psi\,\psi'\,dx=\tfrac{1}{2}[\psi^2]_{L}^{\infty}
=-\tfrac{1}{2}\psi(L)^2=0$.
Consequently all boundary contributions vanish and the energy equation reduces to
\[
    \frac{d}{dt}E(t)
    =-\gamma\|\mathbf{U}(t)\|^2_{L^2(\mathcal{G})}
    =-2\gamma\,E(t).
\]
An application of Gr\"onwall's lemma to the differential inequality
$E'(t)\le -2\gamma E(t)$ gives the conclusion.
\end{proof}

\begin{remark}
The vanishing of all boundary terms is not a coincidence. It is a direct
manifestation of the fact that $A_{Y_0,R_1}$ generates a $C_0$-semigroup of
contractions on $L^2(\mathcal{G})$: by Theorem~\ref{T2}, this is equivalent to
the contractivity of $R_1$, which in turn forces the boundary terms arising in
the Green identity to be non-positive. In the particular case at hand they
vanish altogether, reflecting the precise balance between the subspace $Y_0$
and the relation $R_1$ in the definition of the domain. The systematic
exploitation of this connection between the contractivity of the boundary
relation, the structure of the domain, and the resulting energy dissipation is part of an ongoing work for the KdV equation on general metric graphs.
\end{remark}
\section{The Schr\"odinger operator on  looping-edge and $\mathcal{T}$-shaped  graphs}

In this section we employ the same techniques  used in the case of the free Airy operator to characterize the extensions of the Schr\"odinger operator on a looping-edge and a graph  determined by the edges $e_0=[-L,L]$ and $e_j=[L,\infty)$ considering $-L$ as a {\it terminal vertex}. This strategy will  allow us in the future to study the dynamics of NLS model \eqref{NLS} specially on $\mathcal T$-graphs.

Let us now consider the Schr\"odinger operator  
\begin{equation}
    \label{schr}
    \mathcal{H}_0:\{u_e\}_{e\in E}\mapsto \left\{-\frac{d^2}{dx_e^2} u_e\right\}_{e\in E},
\end{equation}
seen as a densely defined operator on $L^2(\mathcal{G})$ with domain $D(\mathcal{H}_0)$ defined on \eqref{cinf}. Note we would have $D(\mathcal{H}_0^*)=H^2(\mathcal{G})$.

\subsection{Deficiency indices of the free Schr\"odinger operator}
For the symmetric operator $\mathcal{H}_0$ in \eqref{schr} we have that free solutions of $-\partial_x^2 u_e \pm iu_e=0$ are given as linear combinations of complex exponentials of the form $$u_e=c_1 e^{\frac{(1\pm i)}{\sqrt{2}}x_e}+c_2 e^{-\frac{(1\pm i)}{\sqrt{2}}x_e}.$$ 

We note that $e^{\pm\frac{(1\pm i)}{\sqrt{2}}x}$ belong to $L^2(-L,L)$. On the other hand, on a half-line $[L,\infty)$, it would be square integrable only if the real part of $\pm\frac{(1\pm i)}{\sqrt{2}}$ is negative. Along each edge $e=[L,\infty)$ we would therefore always have that the deficiency indices are ${d_+}_e={d_-}_e=1$.

In a looping-edge graph with $N$ half-lines attached at the vertex we would have for $\mathcal{H}_0$ that the deficiency indices  $d_{\mp}=dim (Ker(\mathcal{H}_0\pm iI))$ satisfy $d_+=d_-=2+N$. 
From the von Neumann and Krein theory, one can always expect self-adjoint extensions of $\mathcal{H}_0$ and therefore unitary dynamics.

\subsection{Boundary systems}\label{sec:bdry_schr}
In this subsection we adapt the ideas developed in the case of the Airy operator, for the Schr\"odinger operator posed on $\mathcal{G}$.  

 \begin{definition}
  We define on the graph of $\mathcal{H}_0^*$, $G(\mathcal{H}_0^*)=\{(\textbf{U},\mathcal{H}_0^* \textbf{U})\mid \textbf{U} \in D(\mathcal{H}_0^*)\}$, the standard skew-symmetric form $\Lambda: G(\mathcal{H}_0^*)\times G(\mathcal{H}_0^*) \to \mathbb{C}, $ defined by \begin{equation}\label{form-}
\Lambda((\textbf{U},\mathcal{H}_0^*\textbf{U}),(\textbf{V},\mathcal{H}_0^*\textbf{V})):=[\mathcal{H}_0^*\textbf{U}, \textbf{V}]-[\textbf{U}, \mathcal{H}_0^*\textbf{V}]. 
\end{equation} 
\end{definition}

\begin{definition}\label{def:omega}
   Let $\mathcal{J}=\mathbb{C}^{N+2}$. Define the standard skew-symmetric sesquilinear form on
    $\mathcal{J}\oplus\mathcal{J}$ by
    \begin{equation}\label{eq:omega}
        \rho\bigl((x_1,y_1),(x_2,y_2)\bigr)
        :=\langle y_1,x_2\rangle_{\mathcal{J}}
         -\langle x_1,y_2\rangle_{\mathcal{J}},
        \qquad (x_i,y_i)\in\mathcal{J}\oplus\mathcal{J},
    \end{equation}
    where $\langle\,\cdot\,,\,\cdot\,\rangle_{\mathcal{J}}$ is the
    standard inner product on $\mathbb{C}^{N+2}$.
\end{definition}

\begin{definition}
    We define $\mathbb{F}$ to be the surjective boundary map
    \begin{equation}
        \begin{split}
            \mathbb{F}:&\ G(\mathcal{H}_0^*)\to \mathcal{J}\oplus \mathcal{J}\\
            &(\mathbf{U},\mathcal{H}_0^*\mathbf{U})\mapsto \bigl(\Gamma_0\mathbf{U},\,\Gamma_1\mathbf{U}\bigr),
        \end{split}
    \end{equation}
    where 
    \begin{equation*}
        \Gamma_0\mathbf{U}=\begin{pmatrix}\phi(-L)\\ \phi(L)\\ \psi_1(L)\\ \vdots\\ \psi_N(L)\end{pmatrix},
        \qquad
        \Gamma_1\mathbf{U}=\begin{pmatrix}\phi'(-L)\\ -\phi'(L)\\ \psi_1'(L)\\ \vdots\\ \psi_N'(L)\end{pmatrix}.
    \end{equation*}
\end{definition}

\begin{proposition}\label{prop:boundary_system}
    The quintuple $(\Lambda,\mathcal{J},\mathcal{J},\mathbb{F},\rho)$ is a \emph{boundary system} for $\mathcal{H}_0$, that is,
    $$\Lambda((\textbf{U},\mathcal{H}_0^*\textbf{U}),(\textbf{V},\mathcal{H}_0^*\textbf{V
}))
                =\rho\bigl(\mathbb{F}(\mathbf{U},\mathcal{H}_0^*\mathbf{U}),\,
                         \mathbb{F}(\mathbf{V},\mathcal{H}_0^*\mathbf{V})\bigr).$$
\end{proposition}

\begin{proof}
    The proof follows by integration by parts.
\end{proof}

\subsection{Self-adjoint extensions via boundary systems}

The following theorem, taken from \cite[Theorem~3.3]{Schu2015}, gives a complete
parametrization of all self-adjoint extensions of $\mathcal{H}_0$ in terms of
the boundary system constructed in the previous subsection.

\begin{theorem}\label{thm:SAext}
    An operator $\mathcal{H}$ is a self-adjoint extension of $\mathcal{H}_0$ if
    and only if there exist a subspace $X\subset\mathcal{J}=\mathbb{C}^{N+2}$
    and a self-adjoint operator $\mathcal{L}:X\to X$ such that
    $\mathcal{H}=\mathcal{H}_{X,\mathcal{L}}$, where
    \begin{equation}\label{eq:SAext}
        D(\mathcal{H}_{X,\mathcal{L}})=\Bigl\{\mathbf{U}\in H^2(\mathcal{G})\;\Big|\;
            \Gamma_0\mathbf{U}\in X, \quad \mathcal{L}\Gamma_0\mathbf{U}=P_X\Gamma_1\mathbf{U}
            \Bigr\}.
    \end{equation} 
    Here $P_X:\mathcal{J}\to X$ is the orthogonal projection onto $X$.
\end{theorem}

Theorem~\ref{thm:SAext} suggests a practical procedure for constructing self-adjoint extensions by prescribing boundary conditions in two stages. First, one translates the desired \(k\) conditions (with \(k < N + 2\)) into conditions on \(\Gamma_0\) and \(\Gamma_1\), which can then be directly incorporated into the definition of a matrix \(M : X \to X\). The second stage consists of completing the remaining (unconstrained) entries of \(M\) so that it becomes self-adjoint, and of determining the resulting conditions imposed by the relation \(M\Gamma_0 \mathbf{U} = P_X \Gamma_1 \mathbf{U}\). We exemplify the procedure next. 

\subsection{Examples}

\begin{example}[Self-adjoint extensions on a looping-edge graph]\label{ex:8}

We seek all self-adjoint extensions of $\mathcal{H}_0$ whose domain
includes only the condition looping condition $\phi(-L)=\phi(L)$. This condition constrains $\Gamma_0\mathbf{U}$ in a way that the condition $\Gamma_0\mathbf{U}\in X$ translates into the choice
\[
    X :=\bigl\{x\in\mathcal{J}\;\big|\;x_1=x_2\bigr\},
    \qquad
    X^{\perp}=\operatorname{span}\{e_1-e_2\},
    \qquad
    \dim X = N+1,
\]
with orthogonal projection $P_X v = v-\tfrac{v_1-v_2}{2}(e_1-e_2)$.
Since no further constraint is imposed, $\mathcal{L}:X\to X$ is a
free self-adjoint operator on $X$.

\medskip
Consider the orthonormal basis for $X$
\[
    g_1:=\frac{e_1+e_2}{\sqrt{2}},
    \qquad
    g_{k}:=e_{k+1},\quad k=2,\ldots,N+1.
\]
In this basis the coordinate vectors of $\Gamma_0\mathbf{U}$ and
$P_X\Gamma_1\mathbf{U}$ are, respectively,
\[
    \widetilde\Gamma_0\mathbf{U}
    =\bigl(\sqrt{2}\,\phi(-L),\;\psi_1(L),\;\ldots,\;\psi_N(L)\bigr)^{\!\top},
    \qquad
    \widetilde\Gamma_1\mathbf{U}
    =\Bigl(\tfrac{\phi'(-L)-\phi'(L)}{\sqrt{2}},\;
           \psi_1'(L),\;\ldots,\;\psi_N'(L)\Bigr)^{\!\top}.
\]
Write the matrix of $\mathcal{L}$ in the given basis in the following block form \begin{equation}\label{eq:M-full}
    M=\begin{pmatrix}a & \mathbf{b}^{*}\\\mathbf{b} & M'\end{pmatrix},
    \qquad
    a\in\mathbb{R},\quad
    \mathbf{b}\in\mathbb{C}^N,\quad
    M'\in M_N(\mathbb{C})\text{ Hermitian},
\end{equation}
where $\mathbf{b}^*$ denotes the conjugate transpose of $\mathbf{b}$.
This constitutes an $(N+1)^2$-real-parameter family.

\medskip
Expanding $M\,\widetilde\Gamma_0\mathbf{U}=\widetilde\Gamma_1\mathbf{U}$
row by row one gets
\begin{align}
    \phi'(-L)-\phi'(L)
        &= 2a\,\phi(-L)+\sqrt{2}\,\mathbf{b}^{*}\vec\psi(L),
        \label{eq:loop-bc}\\[6pt]
    \vec\psi\,'(L)
        &= \sqrt{2}\,\mathbf{b}\,\phi(-L)+M'\vec\psi(L),
        \label{eq:half-line-bc}
\end{align}
where $\vec\psi(L)=(\psi_1(L),\ldots,\psi_N(L))^{\!\top}$ and
$\vec\psi\,'(L)=(\psi_1'(L),\ldots,\psi_N'(L))^{\!\top}$.

\begin{theorem}\label{thm:single-loop-ext}
    The self-adjoint extensions of $\mathcal{H}_0$ whose domain satisfies
    the looping condition $\phi(L)=\phi(-L)$ form an $(N+1)^2$-real-parameter family
    $\{\mathcal{H}_{a,\mathbf{b},M'}\}$, parametrised by
    $a\in\mathbb{R}$, $\mathbf{b}\in\mathbb{C}^N$, and a Hermitian
    $M'\in M_N(\mathbb{C})$ with domain 
    \[
        D(\mathcal{H}_{a,\mathbf{b},M'})
        =\Bigl\{\mathbf{U}\in H^2(\mathcal{G})\;\Big|\;
            \phi(-L)=\phi(L),\;
            \eqref{eq:loop-bc},\;
            \eqref{eq:half-line-bc}
        \Bigr\}.
    \]
\end{theorem}

\begin{remark}
    The $2N+1$ real parameters $(a,\operatorname{Re}\mathbf{b},
    \operatorname{Im}\mathbf{b})$ account for two effects. The vector $\mathbf{b}\in \C^N$ couples the loop to the half-lines while the scalar $a\in\mathbb{R}$ introduces a 'Robin-type' condition
              on the difference $\phi'(-L)-\phi'(L)$.
 Note in particular, setting
    $a=0$ and $\mathbf{b}=\mathbf{0}$ in \eqref{eq:loop-bc}--\eqref{eq:half-line-bc}
    induces the pure periodic condition $\phi'(-L)=\phi'(L)$  together with $\vec\psi\,'(L)=M'\vec\psi(L)$ obtaining the fully decoupled self-adjoint extensions $(\mathcal{H}_{M'},D(\mathcal{H}_{M'}))$ with \begin{equation}
        \label{eq:decoupled_dom}D(\mathcal{H}_{M'})=\Bigl\{\mathbf{U}\in H^2(\mathcal{G})\;\Big|\;
            \phi(-L)=\phi(L),\quad
            \phi'(-L)=\phi'(L),\quad
            M'\vec\psi(L)=\vec\psi\,'(L)
        \Bigr\}.
    \end{equation} 

    Several relevant extensions arise as special cases of $\mathcal{H}_{M'}$:
    \begin{itemize}
        \item $M'=0$: the loop satisfies periodic boundary conditions and every
              half-line satisfies a Neumann condition $\psi_j'(L)=0$ at the vertex.
        \item $M'=\alpha I_N$, $\alpha\in\mathbb{R}$: uniform Robin condition on the half-lines,
    $\psi_j'(L)=\alpha\,\psi_j(L)$ for all $j$.
    \item For a real non-zero parameter $Z$, the particular choice $M'=\tfrac{1}{Z}\mathbf{1}\mathbf{1}^T$ (all entries equal to $\tfrac{1}{Z}$) is of particular
interest for future works. Opening $\tfrac{1}{Z}\mathbf{1}\mathbf{1}^T\vec\psi(L)=\vec\psi'(L)$ we get
\begin{equation}
    \label{eq:sum_cond}
    \psi_1'(L)=\cdots=\psi_N'(L), \qquad \sum_{j=1}^N\psi_j(L)=Z\psi_1'(L).
\end{equation}
We obtain then the one-parameter family of decoupled self-adjoint extensions $\mathcal{H}_Z$ with domain \begin{equation}
    \label{eq:decoupled_delta'}
     D(\mathcal{H}_{Z})=\Bigl\{\mathbf{U}\in H^2(\mathcal{G})\;\Big|\;
            \phi(-L)=\phi(L),\quad
            \phi'(-L)=\phi'(L),\quad
            \eqref{eq:sum_cond}
        \Bigr\}.
\end{equation}
    \end{itemize}
\end{remark}    
\end{example}

\subsection{Extensions generating unitary dynamics parameterized by linear subspaces}

We now emulate the strategy of separating derivatives used for the Airy operator. We note in this case there is no actual separation of derivatives rather a simple reference frame to build some of the self-adjoint realizations. 

\begin{proposition}
    For any $\textbf{U},\textbf{V}\in D(\mathcal{H}_0^*)$ we have \begin{equation}\label{3333}
        [\mathcal{H}_0^*\textbf{U},\textbf{V}]-[\textbf{U},\mathcal{H}_0^*\textbf{V}]=\left( Q \vec{\textbf{U}}'\mid \vec{\textbf{V}} \right)-\left(Q\vec{\textbf{U}} \mid \vec{\textbf{V}}'\right)=\left( Q \vec{\textbf{U}}'\mid \vec{\textbf{V}} \right)-\left(\vec{\textbf{U}} \mid Q \vec{\textbf{V}}'\right),
    \end{equation}
    where $$Q=\operatorname{diag}(1,-1,1,\dots,1).$$
\end{proposition}

\begin{definition}
    Given a subspace $Y\subset \C^{2+N}$ consider the operator $\mathcal{H}_Y$, acting as $\mathcal{H}_Y\textbf{U}=\mathcal{H}_0\textbf{U}$, with domain
    \begin{equation}
        D(\mathcal{H}_Y):=\left\{ \textbf{U}\in D(\mathcal{H}_0^*)\mid\ \vec{\textbf{U}}\in Y \ \mbox{and } Q\vec{\textbf{U}}'\in Y^\perp\right\}.
    \end{equation}
\end{definition}
The next proposition determine explicitly the adjoint operator of $\mathcal{H}_Y$. 
\begin{proposition}\label{p55}
   For any subspace $Y\subset \C^{2+N}$ the operator $\mathcal{H}_Y$ is a self-adjoint extension of $\mathcal{H}_0$.  
\end{proposition}
\begin{proof}
    It is enough to prove that the adjoint of $\mathcal{H}_Y$ is given by $\mathcal{H}_Y^* \textbf{V}=\mathcal{H}_0^*\textbf{V}$ with domain \begin{equation}
    D(\mathcal{H}_Y^*):=\left\{\textbf{V}\in D(\mathcal{H}_0^*)\mid \ \vec{\textbf{V}}\in Y\ \mbox{and } Q\vec{\textbf{V}}'\in Y^\perp \right\}=D(\mathcal{H}_Y).
\end{equation}
It is obvious from the definition of $D(\mathcal{H}_Y)$ that if $Q\vec{\textbf{V}}'\in Y^\perp$ and $\vec{\textbf{V}}\in Y$ then $\textbf{V}\in D(\mathcal{H}_Y^*)$. 

Let us prove the converse. Let $\textbf{V}\in D(\mathcal{H}_0^*)$. By definition $\textbf{V}$ is in $D(\mathcal{H}_Y^*)$ if and only if $[\mathcal{H}_Y\textbf{U},\textbf{V}]=[\textbf{U},\mathcal{H}_Y^*\textbf{V}]$ for all $\textbf{U}\in D(\mathcal{H}_Y)$. 

Let $x\in Y^\perp$ be arbitrary. There exists $\textbf{U}\in D(\mathcal{H}_Y)$ such that $\vec{\textbf{U}}=0$ and $Q\vec{\textbf{U}}'=x$. From \eqref{3333}, we must have $$\left(x \mid \vec{\textbf{V}}\right)=\left( Q\vec{\textbf{U}}'\mid \vec{\textbf{V}}\right)=0.$$
Since $x$ was arbitrary, it must be the case $\vec{\textbf{V}}\in (Y^\perp)^\perp=Y$.

Similarly, for $x\in Y$ there exists $\textbf{U}\in D(\mathcal{H}_Y)$ such that $\vec{\textbf{U}}'=0$ and $\vec{\textbf{U}}=x$. From \eqref{3333} we conclude it must be the case 
\begin{equation*}
    \left( x\mid Q\vec{\textbf{V}}'\right)=\left(\vec{\textbf{U}}\mid Q\vec{\textbf{V}}'\right)=0.
\end{equation*}
We therefore have $Q\vec{\textbf{V}}'\in Y^\perp$.
\end{proof}
Let us give some examples on the tadpole graph. Assume $N=1$.

\begin{example}
    Consider the space $Y_0=\mbox{span}\{(1,1,0)^T\}$. Note $\vec{\textbf{U}}\in Y_0$ if and only if $\phi(-L)=\phi(L)$ and $\psi(L)=0$. Also, $Q\vec{\textbf{U}}'\in Y^\perp$ if and only if $\phi'(-L)=\phi'(L)$. We therefore have that $$D(\mathcal{H}_{Y_0})=\{u\in D(\mathcal{H}_0^*)\mid \ \phi(-L)=\phi(L),\ \psi(L)=0\, \mbox{and }\phi'(L)=\phi'(-L)\}.$$
\end{example}

\subsection{The $\mathcal{T}$-shaped graph case}
The abstract theory developed for the looping-edge graph $\mathcal{G}$ can be extended to address $\mathcal{T}$-shaped graphs (see Figure~\ref{fig:t}). These graphs are likewise characterized by a metric graph structure with $V=\{-L, L\}$ and $E=\{[-L, L], [L, \infty), \dots, [L, \infty)\}$. In this setting, the vertex $-L$ is treated as an independent terminal node, and the looping condition $\phi(-L)=\phi(L)$ is not imposed.

\begin{figure}
    \centering
    \includegraphics[width=0.4\linewidth]{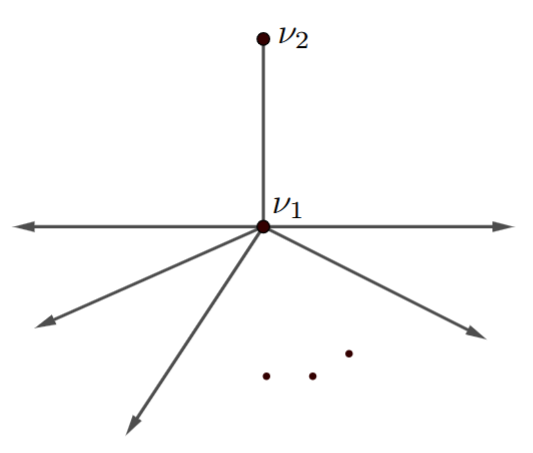}
    \caption{A $\mathcal{T}$-shaped graph}
    \label{fig:t}
\end{figure}

Via the theory of boundary systems developed in Section~\ref{sec:bdry_schr}, we characterize in the following example extension with a derivative coupling at the common vertex. 

\begin{example}[Extensions with prescribed continuity on the derivatives]\label{ex:10}

We seek all self-adjoint extensions of $\mathcal{H}_0$ whose domain includes the
condition
\begin{equation}\label{eq:Neumann-cond}
    \phi'(L)=\psi_1'(L)=\cdots=\psi_N'(L).
\end{equation}

From the boundary map definition and because of the sign in the second component, the conditions
$\phi'(L)=\psi_j'(L)$ are equivalent to
\[
    (\Gamma_1\mathbf{U})_2 + (\Gamma_1\mathbf{U})_{j+2} = 0,
    \qquad j=1,\ldots,N,
\]
or, setting $v_j := e_2+e_{j+2}\in\mathcal{J}$,
\begin{equation}\label{eq:structural}
    v_j^{\,*}\,\Gamma_1\mathbf{U}=0,
    \qquad j=1,\ldots,N.
\end{equation}

\medskip
Since the prescribed conditions place no restriction on $\Gamma_0\mathbf{U}$,
we take the full space $X=\mathcal{J}$ in Theorem~\ref{thm:SAext}.
The domain condition then reads $\mathcal{L}\,\Gamma_0\mathbf{U}=\Gamma_1\mathbf{U}$
for all $\mathbf{U}\in D(\mathcal{H}_{X,\mathcal{L}})$.
Substituting into \eqref{eq:structural},
\[
    v_j^{\,*}\,\mathcal{L}\,\Gamma_0\mathbf{U}=0
    \qquad\text{for all }\Gamma_0\mathbf{U}\in\mathcal{J},
\]
which must hold \emph{structurally}, i.e.\ as a condition on the rows of
the matrix $\mathcal{L}\in M_{N+2}(\mathbb{C})$:
\begin{equation}\label{eq:row-constraint}
    \mathrm{row}_{j+2}(\mathcal{L}) = -\,\mathrm{row}_2(\mathcal{L}),
    \qquad j=1,\ldots,N.
\end{equation}

\medskip
We require $\mathcal{L}=\mathcal{L}^*$.
The row constraints \eqref{eq:row-constraint} propagate to column constraints via
Hermiticity: $\mathrm{col}_{j+2}(\mathcal{L})=-\,\mathrm{col}_2(\mathcal{L})$ for
$j=1,\ldots,N$.
Enforcing both simultaneously and reading off the diagonal entries (which must be real),
the most general Hermitian matrix satisfying \eqref{eq:row-constraint} is

\begin{equation}\label{eq:L-matrix}
    \mathcal{L}=
    \begin{pmatrix}
        a      & \beta      & -\beta & \cdots & -\beta \\[4pt]
        \bar\beta & c       & -c     & \cdots & -c     \\[2pt]
        -\bar\beta & -c     & c      & \cdots & c      \\
        \vdots  & \vdots    & \vdots & \ddots & \vdots \\
        -\bar\beta & -c     & c      & \cdots & c
    \end{pmatrix},
    \qquad
    a,\,c\in\mathbb{R},\quad\beta\in\mathbb{C},
\end{equation}
where the lower-right $N\times N$ block is the constant matrix with every entry equal to $c$.
This constitutes a \emph{four-real-parameter family} $(a,c,\operatorname{Re}\beta,\operatorname{Im}\beta)$.

Expanding $\mathcal{L}\,\Gamma_0\mathbf{U}=\Gamma_1\mathbf{U}$ row by row and introducing the shorthand
\[
    \sigma := \phi(L)-\sum_{j=1}^N\psi_j(L),
\]
one obtains the following two independent conditions:
\begin{align}
    \phi'(-L) &= a\,\phi(-L)+\beta\,\sigma,
    \label{eq:bc1}\\[4pt]
    -\phi'(L) &= \bar\beta\,\phi(-L)+c\,\sigma.
    \label{eq:bc2}
\end{align}
The remaining $N$ equations (rows $3$ through $N+2$) each read
$\psi_j'(L)=\bar\beta\,\phi(-L)+c\,\sigma$, which is exactly $-$\eqref{eq:bc2},
thereby reproducing the prescribed condition~\eqref{eq:Neumann-cond} automatically.

\begin{theorem}
    The self-adjoint extensions of $\mathcal{H}_0$ whose domain satisfies
    \eqref{eq:Neumann-cond} form the four-real-parameter family
    $\{\mathcal{H}_{a,c,\beta}\}_{a,c\in\mathbb{R},\,\beta\in\mathbb{C}}$,
    where
    \[
        D(\mathcal{H}_{a,c,\beta})
        =\Bigl\{\mathbf{U}\in H^2(\mathcal{G})\;\Big|\;
            \eqref{eq:Neumann-cond},\;\eqref{eq:bc1},\;\eqref{eq:bc2}
        \Bigr\}.
    \]
\end{theorem}

\begin{remark}
    Setting $\beta=0$ decouples the left endpoint from the vertex:
    \eqref{eq:bc1} becomes a standard Robin condition $\phi'(-L)=a\,\phi(-L)$,
    and \eqref{eq:bc2} becomes a Robin condition at the vertex controlled by $c$
    and the combination $\sigma$. More precisely, for parameters $Z_1\in \R$, $0\neq Z_2\in \R$ we have the self-adjoint extension $(H_{Z_1,Z_2},D(H_{Z_1,Z_2}))$ with \[
     D(\mathcal{H}_{Z_1,Z_2})
        =\Bigl\{\mathbf{U}\in H^2(\mathcal{G})\;\Big|\;
            \eqref{eq:Neumann-cond},\;\phi'(-L)=Z_1\phi(-L),\;\  \sum_{j=1}^N \psi_j(L)=\phi(L)+Z_2\phi'(L)
        \Bigr\}.
    \]
    The choice $a=c=\beta=0$ yields the pure Neumann extension
    $\phi'(-L)=0$, $\phi'(L)=\psi_j'(L)=0$ for all $j$.
\end{remark}
\end{example}

\subsection{A simple application to instability of standing waves of the
cubic NLS}\label{sec:application}

Let $D_{\mathrm{ext}}=D(\mathcal{H})$ where $D(\mathcal{H})$ is defined
in \eqref{eq:decoupled_dom} with $M'=0$. That is, $D_{\mathrm{ext}}$
defines a self-adjoint extension of the Laplacian inducing pure periodic
conditions on the loop and independent pure Neumann conditions on each
half-line.

We seek real-valued standing waves
$\bU(t,x) = e^{i\omega t}\bPhi_\omega(x)$
with $\bPhi_\omega = (\phi_\omega,\psi_{1,\omega},\dots,\psi_{N,\omega})
\in D_{\mathrm{ext}}$.
Substituting into \eqref{NLS} with $p=1$, each component of $\bPhi_\omega$
satisfies an independent stationary NLS equation $f''+f^3-\omega f=0$ on
its respective edge.

For the loop component we consider a positive, real, $2L$-periodic solution
constructed via the Jacobi dnoidal function. By Theorem~2.1 of \cite{An3},
for $\omega>\omega^*:=\tfrac{\pi^2}{2L^2}$, the function
\begin{equation}\label{d1}
  \phi_\omega(x)=\eta_1\, \mathrm{dn}\!\left(\frac{\eta_1}{\sqrt{2}}\,x;\,k\right)
\end{equation}
with elliptic modulus $k\in (0,1)$ determined by
\begin{equation}\label{d2}
  k^2(\eta,\omega)=\frac{2\omega-2\eta^2_2}{2\omega-\eta^2_2},\qquad
  \eta^2_1+\eta^2_2=2\omega,\qquad 0<\eta_2<\eta_1,
\end{equation}
solves the stationary loop equation. The half-line stationary equation is
solved by the \emph{half-soliton}
\[
  \psi_{\omega}(x)
  =\sqrt{2\omega}\,\sech\!\bigl(\sqrt{\omega}(x-L)\bigr),
  \qquad x\in[L,+\infty),
\]
which is the restriction to $[L,+\infty)$ of the full soliton centred at
$x=L$. The Neumann condition $\psi_{\omega}'(L)=0$ holds automatically
because the full soliton is even about $L$. For $j=1,\dots,n\leq N$ we
therefore set $\psi_{j,\omega}\equiv\psi_{\omega}$, and for
$j=n+1,\dots,N$ we set $\psi_{j,\omega}\equiv 0$. The standing wave
profile under study is thus (see Figure~\ref{fig:profile})
\begin{equation}\label{eq:profile}
  \bPhi_{\omega,n}\equiv\bPhi_\omega
  = \bigl(\phi_\omega,\,
    \underbrace{\psi_{1,\omega},\dots,\psi_{n,\omega}}_{n\text{ half-solitons}},\,
    \underbrace{0,\dots,0}_{N-n\text{ zero-solitons}}\bigr)
  \in D_{\mathrm{ext}},
  \qquad\omega>\omega^*.
\end{equation}

\begin{figure}
    \centering
    \includegraphics[width=0.95\linewidth]{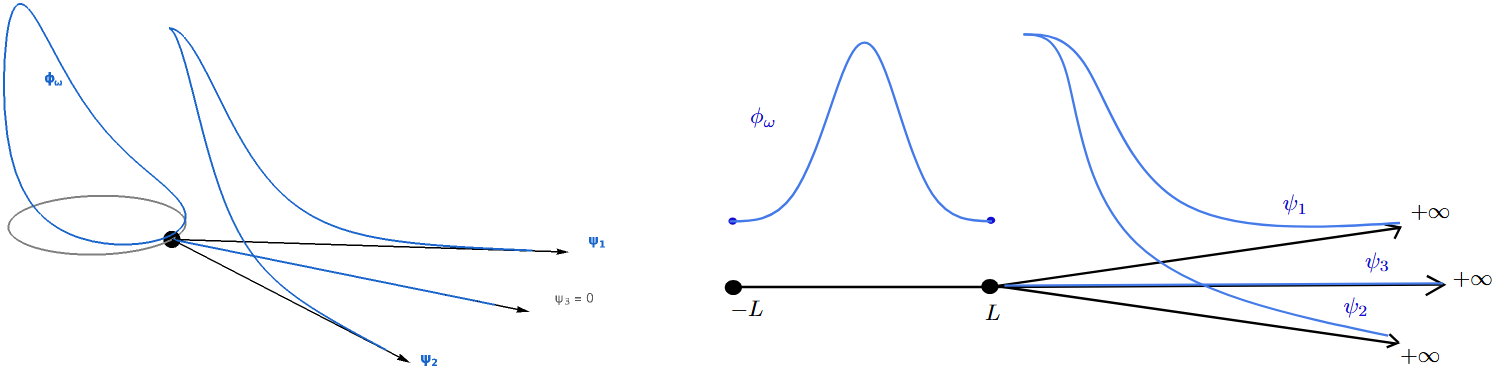}
    \caption{Schematic of the standing wave profile $\Phi_{\omega,n}$ with $N=3$ and $n=2$.}
    \label{fig:profile}
\end{figure}

\begin{theorem}
\label{thm:phase-drift}
Let $L>0$ and $N\ge 1$. For every $n\in\{1,\dots,N\}$ and every
$\omega>\omega_*$, the orbit
\[
  \mathcal{O}(\mathbf{\Phi}_\omega)
  = \bigl\{e^{i\theta}\bigl(\phi_\omega(\cdot+s),\,
    \psi_{1,\omega},\dots,\psi_{n,\omega},0,\dots,0\bigr)
    :\theta\in\mathbb{R},\;s\in\mathbb{R}/(2L\mathbb{Z})\bigr\}
\]
is orbitally unstable in $H^1(\mathcal{G})$.
\end{theorem}

\begin{proof}
The decoupled structure of $D_{\mathrm{ext}}$ reduces global well-posedness
on $H^1(\mathcal{G})$ to independent problems on each edge. On the loop
$[-L,L]$ with periodic boundary conditions, global well-posedness in $H^1$
is classical and follows from semigroup theory together with conservation of
energy and charge. On each half-line $[L,+\infty)$
with Neumann boundary condition $\psi'(L)=0$, well-posedness is reduced to
the full-line case via the reflection method: extending $\psi$ to an even
function on $\mathbb{R}$ embeds the problem into the cubic NLS on
$\mathbb{R}$ restricted to even initial data, which is invariant under the
flow. Global solutions for all $t\in\mathbb{R}$ are then obtained from the
conservation of energy and charge via the Gagliardo--Nirenberg inequality.

\medskip
Fix $\varepsilon>0$ and define the initial datum
\[
  \mathbf{U}_\varepsilon(0)
  := \bigl(\phi_\omega,\;\psi_{1,\omega+\varepsilon},\;
     \psi_{2,\omega},\dots,\psi_{n,\omega},\;0,\dots,0\bigr),
\]
where
$\psi_{1,\omega+\varepsilon}(x)
  = \sqrt{2(\omega+\varepsilon)}\,
    \operatorname{sech}\!\bigl(\sqrt{\omega+\varepsilon}(x-L)\bigr)$
is the half-soliton at the shifted frequency $\omega+\varepsilon$.
Since $\omega\mapsto\psi_{1,\omega}$ is $C^1$ as a map into
$H^1([L,+\infty))$, the mean value inequality in Banach spaces
gives
\begin{equation}\label{eq:init-dist}
  \|\mathbf{U}_\varepsilon(0)-\mathbf{\Phi}_\omega\|_{H^1(\mathcal{G})}
  = \|\psi_{1,\omega+\varepsilon}-\psi_{1,\omega}\|_{H^1([L,+\infty))}
  \le C_\omega\,\varepsilon,
\end{equation}
where
$C_\omega
  = \sup_{\omega'\in(\omega,\,\omega+1)}
    \|\partial_\omega\psi_{1,\omega'}\|_{H^1([L,+\infty))}<\infty$.
In particular,
$\mathrm{dist}_{H^1}(\mathbf{U}_\varepsilon(0),\mathcal{O}(\mathbf{\Phi}_\omega))
\to 0$ as $\varepsilon\to 0$.

\medskip
Since every component of $\mathbf{U}_\varepsilon(0)$ is an exact stationary
profile at its respective frequency, each component satisfies its own NLS
equation with its boundary conditions. By the uniqueness of $H^1$ solutions,
the global solution is therefore given by
\begin{equation}\label{eq:exact-sol}
  \mathbf{U}_\varepsilon(t)
  = \bigl(
      e^{i\omega t}\phi_\omega,\;
      e^{i(\omega+\varepsilon)t}\psi_{1,\omega+\varepsilon},\;
      e^{i\omega t}\psi_{2,\omega},\;
      \dots,\;
      e^{i\omega t}\psi_{n,\omega},\;
      0,\dots,0
    \bigr).
\end{equation}
At the time $t_\varepsilon:=\pi/\varepsilon$ one has
$e^{i(\omega+\varepsilon)t_\varepsilon}
= e^{i\omega t_\varepsilon}\cdot e^{i\pi}
= -e^{i\omega t_\varepsilon}$,
so the solution becomes
\begin{equation}\label{eq:sign-flip}
  \mathbf{U}_\varepsilon(t_\varepsilon)
  = e^{i\omega t_\varepsilon}
    \bigl(\phi_\omega,\;-\psi_{1,\omega+\varepsilon},\;
     \psi_{2,\omega},\dots,\psi_{n,\omega},\;0,\dots\bigr),
\end{equation}
that is, the first half-line component has flipped sign relative to all
others.

\medskip
We now derive a lower bound on the distance from the orbit at time
$t_\varepsilon$. Take any orbit element
$\mathbf{V}=e^{i\theta}(\phi_\omega(\cdot+s),\psi_{1,\omega},\dots)
\in\mathcal{O}(\mathbf{\Phi}_\omega)$
and set $\mu:=\theta-\omega t_\varepsilon$. Since the components are
supported on disjoint edges and
$\|(1-e^{i\mu})\psi_{j,\omega}\|_{L^2}^2\geq 0$ for $j=2,\dots,N$,
we obtain
\begin{equation}\label{eq:quad-form}
  \begin{split}
    \|\mathbf{U}_\varepsilon(t_\varepsilon)-\mathbf{V}\|_{L^2(\mathcal{G})}^2
    &\geq
      \|\phi_\omega - e^{i\mu}\phi_\omega(\cdot+s)\|_{L^2}^2
      + \|\psi_{1,\omega+\varepsilon}+e^{i\mu}\psi_{1,\omega}\|_{L^2}^2\\
    &= A + 2\cos(\mu)\bigl(P-q(s)\bigr),
  \end{split}
\end{equation}
where
\begin{equation*}
  \begin{split}
    A &:= 2\|\phi_\omega\|_{L^2}^2
         + \|\psi_{1,\omega+\varepsilon}\|_{L^2}^2
         + \|\psi_{1,\omega}\|_{L^2}^2,\\
    P &:= \langle\psi_{1,\omega+\varepsilon},\psi_{1,\omega}
               \rangle_{L^2([L,+\infty))},\\
    q(s) &:= \langle\phi_\omega,\phi_\omega(\cdot+s)
                 \rangle_{L^2([-L,L])}.
  \end{split}
\end{equation*}
Since $\cos\mu\in[-1,1]$, minimising the right-hand side of
\eqref{eq:quad-form} over $\mu$ and then taking the infimum over
$(\theta,s)$ gives
\begin{equation}\label{eq:inf-orbit}
  \inf_{(\theta,s)}
  \|\mathbf{U}_\varepsilon(t_\varepsilon)-\mathbf{V}\|_{L^2}^2
  \geq A - 2\sup_{s\in\mathbb{R}/(2L\mathbb{Z})}|P-q(s)|.
\end{equation}
It is shown in Appendix~\ref{ap:lower_bound} that
\[
  A - 2\sup_{s}|P-q(s)|
  \geq \min\!\bigl(4\sqrt{\omega},\;2q_{\min}\bigr)
  =: c_0^2 > 0,
\]
where $q_{\min}:=\min_{s}q(s)>0$ and $c_0$ depends only on $\omega$, $L$,
and $k$, but not on $\varepsilon$. Combining with \eqref{eq:inf-orbit} we conclude
\begin{equation}\label{eq:lower-bound}
  \mathrm{dist}_{H^1}(\mathbf{U}_\varepsilon(t_\varepsilon),\,
    \mathcal{O}(\mathbf{\Phi}_\omega))
  \geq
  \mathrm{dist}_{L^2}(\mathbf{U}_\varepsilon(t_\varepsilon),\,
    \mathcal{O}(\mathbf{\Phi}_\omega))
  \geq c_0 > 0.
\end{equation}

\medskip
Finally, given $\delta>0$ arbitrary, one can choose $\varepsilon>0$ small enough that
$C_\omega\varepsilon<\delta$. Then \eqref{eq:init-dist} gives
$\mathrm{dist}_{H^1}(\mathbf{U}_\varepsilon(0),
\mathcal{O}(\mathbf{\Phi}_\omega))<\delta$,
while \eqref{eq:lower-bound} gives
$\mathrm{dist}_{H^1}(\mathbf{U}_\varepsilon(t_\varepsilon),
\mathcal{O}(\mathbf{\Phi}_\omega))\geq c_0$.
This is precisely the opposite of orbital
stability, and the proof is complete.
\end{proof}

\section*{Acknowledgments}

\noindent 
The authors express their gratitude to Prof. Nataliia Goloshchapova (IME-USP) for her invaluable comments on boundary systems, which significantly improved the results presented in the second part of this manuscript.

This study is financed, in part, by the S\~ao Paulo Research Foundation (FAPESP), Brazil. Process number 2024/20623-7. \\
A. Mu\~noz was partially supported by CNPq, Conselho Nacional de Desenvolvimento Cient\'{\i}fico e Tecnol\'ogico - Brazil grant 170246/2023-0 and INCTMat, Instituto Nacional de Ci\^encia e Tecnologia de Matem\'atica - Brazil. 
J. Angulo was partially funded by CNPq/Brazil Grant.

\appendix

\section{Lower bound computation}\label{ap:lower_bound}

In this appendix we establish the inequality
\[
  A - 2\sup_{s\in\mathbb{R}/2L\mathbb{Z}}|P-q(s)|
  \geq \min\!\bigl(4\sqrt{\omega},\;2q_{\min}\bigr)
  =: c_0^2 > 0,
\]
used in Section~\ref{sec:application}.

\medskip
Since $\phi_\omega$ and $\psi_\omega$ are strictly positive, both $P$ and
$q(s)$ are strictly positive. The map $s\mapsto q(s)$ is continuous on the
compact set $\mathbb{R}/2L\mathbb{Z}$, so it attains strictly positive
extrema $0<q_{\min}<q_{\max}$. The maximum is attained at $s=0$, with
$q_{\max}=q(0)=\|\phi_\omega\|^2_{L^2}$, as a consequence of the
Cauchy--Schwarz inequality and the invariance of the $L^2$ norm under
periodic translation. Since $P$ is independent of $s$, the convexity of
the absolute value gives
\[
  \sup_s|P-q(s)|\leq\max\bigl(|P-q_{\min}|,\;|P-q_{\max}|\bigr).
\]
We argue in three cases according to the position of $P$ relative to
$[q_{\min},q_{\max}]$.

\bigskip\noindent
\textbf{Case 1: $P\leq q_{\min}$.}
Since $q(s)\geq q_{\min}\geq P$ for all $s$, one has $|P-q(s)|=q(s)-P$
and the supremum is attained at $s=0$:
\[
  \sup_s|P-q(s)| = q_{\max}-P = \|\phi_\omega\|^2_{L^2}-P.
\]
Substituting back and cancelling the loop norm terms:
\[
  \begin{split}
    A - 2\sup_s|P-q(s)|
    &= 2\|\phi_\omega\|_{L^2}^2
       + \|\psi_{1,\omega+\varepsilon}\|_{L^2}^2
       + \|\psi_{1,\omega}\|_{L^2}^2
       - 2\|\phi_\omega\|_{L^2}^2 + 2P\\
    &= \|\psi_{1,\omega+\varepsilon}\|_{L^2}^2
       + \|\psi_{1,\omega}\|_{L^2}^2 + 2P.
  \end{split}
\]
All three terms are non-negative and
$\|\psi_{1,\omega}\|_{L^2}^2 = 2\sqrt{\omega}$, so
\[
  A - 2\sup_s|P-q(s)|
  \geq 2\|\psi_{1,\omega}\|_{L^2}^2
  = 4\sqrt{\omega} > 0.
\]

\bigskip\noindent
\textbf{Case 2: $q_{\min}\leq P\leq q_{\max}$.}
Here $|P-q_{\min}|=P-q_{\min}$ and $|P-q_{\max}|=q_{\max}-P$, so
\[
  \sup_s|P-q(s)| = \max(P-q_{\min},\;q_{\max}-P).
\]

\smallskip\noindent
\textit{Subcase~2a: $q_{\max}-P\geq P-q_{\min}$.}
The supremum equals $q_{\max}-P=\|\phi_\omega\|_{L^2}^2-P$, and the same
computation as in Case~1 yields
\[
  A-2\sup_s|P-q(s)|
  = \|\psi_{1,\omega+\varepsilon}\|_{L^2}^2
    +\|\psi_{1,\omega}\|_{L^2}^2
    +2P
  \geq 4\sqrt{\omega} > 0.
\]

\smallskip\noindent
\textit{Subcase~2b: $P-q_{\min}>q_{\max}-P$.}
The supremum equals $P-q_{\min}$. Since
$P\leq q_{\max}=\|\phi_\omega\|_{L^2}^2$, we have
$2\|\phi_\omega\|_{L^2}^2-2P\geq 0$, and therefore
\[
  \begin{split}
    A-2\sup_s|P-q(s)|
    &= 2\|\phi_\omega\|_{L^2}^2
       + \|\psi_{1,\omega+\varepsilon}\|_{L^2}^2
       + \|\psi_{1,\omega}\|_{L^2}^2
       - 2P + 2q_{\min}\\
    &\geq \|\psi_{1,\omega+\varepsilon}\|_{L^2}^2
          + \|\psi_{1,\omega}\|_{L^2}^2 + 2q_{\min}
     \geq 2q_{\min} > 0.
  \end{split}
\]

Combining both subcases, in Case~2:
\[
  A-2\sup_s|P-q(s)|\geq\min\bigl(4\sqrt{\omega},\;2q_{\min}\bigr)>0.
\]

\bigskip\noindent
\textbf{Case 3: $P>q_{\max}$.}
Since $P>q(s)$ for all $s$, one has $|P-q(s)|=P-q(s)$ and the supremum
is attained at the minimizer of $q$:
\[
  \sup_s|P-q(s)| = P - q_{\min}.
\]
By the Cauchy--Schwarz inequality,
$P\leq\|\psi_{1,\omega+\varepsilon}\|_{L^2}\|\psi_{1,\omega}\|_{L^2}$,
hence
\[
  \|\psi_{1,\omega+\varepsilon}\|_{L^2}^2
  + \|\psi_{1,\omega}\|_{L^2}^2 - 2P
  \geq
  \bigl(\|\psi_{1,\omega+\varepsilon}\|_{L^2}
        - \|\psi_{1,\omega}\|_{L^2}\bigr)^2
  \geq 0.
\]
Therefore:
\[
  \begin{split}
    A - 2\sup_s|P-q(s)|
    &= 2\|\phi_\omega\|_{L^2}^2 + 2q_{\min}
       + \bigl(\|\psi_{1,\omega+\varepsilon}\|_{L^2}^2
               + \|\psi_{1,\omega}\|_{L^2}^2 - 2P\bigr)\\
    &\geq 2\|\phi_\omega\|_{L^2}^2 + 2q_{\min}
     \geq 2q_{\min} > 0.
  \end{split}
\]

\medskip\noindent
Collecting all cases, we conclude that
\[
  A - 2\sup_{s\in\mathbb{R}/2L\mathbb{Z}}|P-q(s)|
  \geq \min\!\bigl(4\sqrt{\omega},\;2q_{\min}\bigr)
  =: c_0^2 > 0,
\]
where $c_0^2$ depends only on $\omega$, $L$, and the elliptic modulus
$k=k(\omega)$, but not on $\varepsilon$.


\begin{thebibliography}{99}
    
\bibitem{AdaNoj15}
R. Adami, C. Cacciapuoti, D. Finco, and D. Noja,  
{\it Stable standing waves for a NLS on star graphs as local 
minimizers of the constrained energy,}
 J. Differential Equations, 260 (2016),  7397--7415.
 
\bibitem{AdaNoj14}
R. Adami, C. Cacciapuoti, D. Finco, and D. Noja, 
{\it Variational
properties and orbital stability of standing waves for NLS
equation on a star graph,}
 J. Differential Equations, 257 (2014),
3738--3777.

\bibitem{AST}
R. Adami, E. Serra, and P. Tilli,
{\it NLS ground states on graphs,}
 Calc. Var. Partial Differential Equations, 54 (2015), no. 1,
743--761.

\bibitem{ASTcri} R. Adami, E. Serra, and P. Tilli,{\it Negative energy ground states for the $L^2$-critical NLSE on metric  graphs,} Comm. Math. Phys. 352  (2017), 387--406.
 
\bibitem{ASTmul}
R. Adami, E. Serra, and P. Tilli,
{\it Multiple positive bound states for the subcritical NLS equation on metric graphs,}
 Calc. Var., 58 (2019), no. 1,5, 16pp.
 
 \bibitem{AmCrep} K. Ammari and E. Cr\'{e}peau,  {\it Feedback Stabilization and Boundary Controllability  of the  Korteweg--de Vries Equation on a Star-Shaped Network}, SIAM Journal on Control and Optimization, 56(3), pp. 1620--1639, (2018)

\bibitem{An1}  J. Angulo, {\it Stability theory for two-lobe states on the tadpole graph for the NLS equation}, Nonlinearity, 37, 045015, (2024)


\bibitem{An2}  J. Angulo, \textit{Stability theory for the NLS on looping edge graphs}, Math. Z.,  308, 19 (2024). 

\bibitem{An3} 
J. Angulo. 
{\it Nonlinear stability of periodic traveling wave solutions to the Schr\"odinger and the modified Korteweg-de Vries equations,} 
JDE, 235, 1--30. 2007.
 
  \bibitem{AC} J. Angulo  and  M. Cavalcante, \textit{Nonlinear Dispersive Equations on Star Graphs}, $32^o$ Col\'oquio Brasileiro de Matem\'atica,  IMPA, (2019).

\bibitem{AC1} J. Angulo  and  M. Cavalcante, \textit{Linear instability of stationary solitons for the Korteweg-de Vries equation on a star graph},  Nonlinearity, 34, (2021), 3373-3410.

\bibitem{AC2} J. Angulo  and  M. Cavalcante, \textit{dynamics of the Korteweg-de Vries equation on a balanced metric graph}. To appear 
in  the Bulletin of the Brazilian Mathematical Society, New Series (BSBM), (2024).





\bibitem{AngGol17a} 
J. Angulo and  N. Goloshchapova, 
{\it On the orbital instability of excited states for the NLS equation with the
$\delta$-interaction on a star graph,} Discrete Contin. Dyn. Syst. A., 38(10), (2018), 5039--5066.

\bibitem{AngGol17b} 
J. Angulo and  N. Goloshchapova, 
{\it Extension theory approach in the stability of the standing waves for the NLS equation with point interactions on a star graph}, Adv. Differential Equations 23 (11-12), (2018), 793--846.

\bibitem{AP1} J. Angulo and R. Plaza, \textit{Instability of static solutions of the sine-Gordon equation
on a $Y$-junction graph with $\delta$-interaction}, J. Nonlinear Science, 31, 50, (2021).
 
\bibitem{AP2} J. Angulo and R. Plaza, \textit{Instability theory of kink and anti-kink profiles for the sine-Gordon on Josephson tricrystal boundaries}, Physica D, v. 427, 133020, (2021).
    
   \bibitem{AP3}  J. Angulo and R. Plaza, \textit{Unstable kink and anti-kink profiles for the sine-Gordon
on a $Y$ -junction graph with $\delta' $-interaction at the vertex}, Mathematische Zeitschrift, 300, 2885--2915 (2022). 


 \bibitem{Ardila}  A. H. Ardila, \textit{ Orbital stability of standing waves for supercritical NLS with potential on graphs}, Applicable
Analysis, v. 99, 8, (2020), 1359-1372.


\bibitem{BK}
G. Berkolaiko  and  P. Kuchment, {\it Introduction to Quantum Graphs},
 Mathematical Surveys and Monographs, 
 186, Amer. Math.
Soc., Providence, RI, 2013.

\bibitem{BMP}
G. Berkolaiko  , J. L. Marzuola, and D. E. Pelinovsky,
{\it Edge-localized states on quantum graphs in the limit of large mass},
 Annales de l'Institut Henri Poincare C, Analyse Non Lineaire, 38 (2021), 1295-1335.

\bibitem{BlaExn08}
J. Blank, P. Exner,  and M. Havlicek, 
{\it Hilbert Space Operators in Quantum Physics},
 2nd edition, Theoretical and Mathematical
Physics, Springer, New York, 2008.


\bibitem{BurCas01}
R. Burioni, D. Cassi, M. Rasetti, P. Sodano, and A. Vezzani,
{\it Bose-Einstein condensation on inhomogeneous
complex networks,}
 J. Phys. B: At. Mol. Opt. Phys., 34 (2001),
4697--4710.




\bibitem{CFN} C. Cacciapuoti, D. Finco, and  D. Noja, {\it Topology induced bifurcations for the NLS on the tadpole graph}, Phys.
Rev. E 91 (2015), 013206.

\bibitem{CFN2} C. Cacciapuoti, D. Finco, and  D. Noja, {\it Ground state and orbital stability for the NLS equation on a general starlike graph with potentials}, Nonlinearity 30, 8,  (2017), 3271-3303.

\bibitem{Cav1} M. Cavalcante, {\it The Korteweg-de Vries equation on a metric star graph}, Z. Angew. Math. Phys. 69, Art. 124 (2018).


\bibitem{Chuiko} G. P. Chuiko, O. V.  Dvornik, S.I.  Shyian, and Y.A. Baganov,  {\it A new age-related model for blood stroke volume.} Computers in Biology and Medicine, 79 (2016) 144--148.

\bibitem{Crepeau} E. Cr\'epeau and M. Sorine,  {\it A reduced model of pulsatile flow in an arterial compartment.} Chaos Solitons
Fractals, 34 (2) (2007), 594--605.

\bibitem{Engel}
Engel K.J. and Nagel R.  {\it One parameter semigroups for linear evolution equations},  Springer New York, (2006).


\bibitem{Ex} P. Exner,  {\it Magnetoresonance on a lasso graph}, Foundations of Physics, v. 27, Article number: 171 (1997).



\bibitem{ExS} P. Exner, P. $\check{S}$eba, {\it  Free quantum motion on a branching graph}, Rep.
Math.Phys. 28 (1989), 7--26.
 
\bibitem{ExSere} P. Exner, E.  $\check{S}$ere$\check{s}$ov\'a, {\it Appendix resonances on a simple graph}, J. Phys. A27 (1994), 8269--8278.
 
 
\bibitem{Fid15} 
F. Fidaleo, 
{\it Harmonic analysis on inhomogeneous
amenable networks and the Bose-Einstein condensation,}
 J. Stat. Phys., 160 (2015), 715--759.
 

\bibitem{GSS1} 
M. Grillakis, J. Shatah and W. Strauss. 
{\it Stability theory of solitary waves in the presence of symmetry, I,} 
J. Funct. Anal., 74(1), 160--197. 1987.


	
\bibitem{KP} A. Kairzhan  and  D. E.  Pelinovsky,  {\it Multi-pulse edge-localized states on quantum graphs},  Analysis and Mathematical Physics, 11, (2021), 171 (26pp). 	
	
\bibitem{KPG} A. Kairzhan, D. E. Pelinovsky, and R. Goodman,  {\it drift of spectrally stable shifted states on star graphs}, SIAM, Journal on Applied Dynamical Systems,  18, (2019), 1723--1755.

\bibitem{KMPX} A. Kairzhan, R. Marangell, D. E.  Pelinovsky, and K. Xiao,  {\it Existence of standing waves on a flower graph}, 
J.  Differential Equations,  271, (2021), 719--763.  

\bibitem{KNP} A. Kairzhan, D. Noja and  D. E.  Pelinovsky,  {\it Standing waves on quantum graphs},  J. Phys. A; Math. Theor.  55, (2022), 243001 (51pp). 


\bibitem{K}
P. Kuchment, 
{\it Quantum graphs, I. Some basic structures,}
Waves Random Media, 14 (2004), 107--128.



\bibitem{MP} J.L. Marzuola and D. E. Pelinovsky, {\it Ground state on the dumbbell graph,} Appl. Math. Res. Express. AMRX. 1,(2016), 98--145.




\bibitem{Mug15} 
D. Mugnolo, 
{\it Mathematical Technology of Networks},
 Bielefeld, December
2013, Springer Proceedings in Mathematics $\&$ Statistics 128,
2015.


\bibitem{MNS} 
D. Mugnolo, D. Noja and C. Seifter, {\it  Airy-type evolution equations on star graphs}, 
Anal. PDE, V. 11,  (2018), 1625-1652.

\bibitem{MNS2} 
D. Mugnolo, D. Noja and C. Seifter, {\it  On solitary waves for the Korteweg-de Vries equation on metric star graphs}, In: Ball, J., Tylli, HO., Virtanen, J.A. (eds) Operator Theory, Related Fields, and Applications. IWOTA 2023. Operator Theory: Advances and Applications, vol 307. Birkhäuser, Cham. (2025).





 \bibitem{Noj14}
D. Noja, {\it  Nonlinear Schr\"odinger equation on graphs: recent results and open problems}, 
Philos. Trans. R. Soc. Lond. Ser. A Math. Phys. Eng. Sci., 372 (2014),
 20130002, 20 pp.
 
 \bibitem{NPS} D. Noja, D. E. Pelinovsky, and G. Shaikhova, {\it Bifurcations and stability of standing waves in the nonlinear Schr\"odinger equation on the tadpole graph}, Nonlinearity 28, (2015), 2343-2378.

\bibitem{AN2016}
F. Natali and J. Angulo.
{\it On the instability of periodic waves for dispersive equations}.
Diff. Int. Eq., 29 (9/10), 837 - 874. 2016.

 
\bibitem{NP}  D. Noja and D. E. Pelinovsky, {\it Standing waves of the quintic NLS equation on the tadpole graph}, Calc. Var. 59, 173 (2020).

\bibitem{Pan}  A. Pankov, {\it Nonlinear Schr\"odinger equations on periodic metric  graphs}, Discrete and Continuous Dynamical Systems-A, 38 (2018), no. 4, 697--714.


\bibitem{PCP} H. Parada, E. Crépeau and C. Prieur, {\it Stability of KdV equation on a network with bounded and unbounded branches}, ESAIM: COCV, 30 , 84 (2024).

 \bibitem{Schu2015} C. Schubert, C. Seifert, J. Voigt and M. Waurick, \textit{Boundary systems and (skew-) self-adjoint operators on infinite metric graphs}, Math. Nachr., 288:1776-1785, (2015).

\bibitem{SBM} Z.A. Sobirov, D. Babajanov,  and D. Matrasulov, {\it Nonlinear standing waves on planar branched systems: Shrinking into metric graph}, Nanosystems,  8 (2017) 29--37.




\end{thebibliography}
\end{document}